\documentclass{imsart}

\RequirePackage[OT1]{fontenc}
\usepackage{xr}
\RequirePackage{amsthm,amsmath,amssymb,natbib}
\usepackage{epsf}
\usepackage{graphicx}
\usepackage{epsfig}

\arxiv{0000.0000}

\startlocaldefs
\numberwithin{equation}{section}
\theoremstyle{remark}

\numberwithin{equation}{section}

\theoremstyle{plain}
\newtheorem{thm}{Theorem}[section]
\newtheorem{remark}{Remark}[section]
\newtheorem{lem}[thm]{Lemma}

\newtheorem{cor}{Corollary}

\theoremstyle{definition}

\newtheorem{ass}{Assumption}[section]

\newcommand {\BB} {{\mathcal B}}
\newcommand {\MM} {{\mathcal M}}
\newcommand {\XX} {{\mathcal X}}
\newcommand {\rr} {{\mathbb R}}
\newcommand {\RR} {{\mathcal R}}
\newcommand {\UU} {{\mathcal U}}

\endlocaldefs

\begin{document}

\begin{frontmatter}
\title{An iterative hard thresholding estimator for low rank matrix recovery with explicit limiting distribution}
\runtitle{An iterative hard thresholding estimator for low rank matrix}

\begin{aug}
\author{\fnms{Alexandra} \snm{Carpentier\thanksref{a1,e1}}\ead[label=e1,mark]{carpentier@math.uni-potsdam.de}}
\and
\author{\fnms{Arlene K. H.} \snm{Kim\thanksref{a,e2}}\corref{}\ead[label=e2,mark]{a.kim@statslab.cam.ac.uk}}
\address[a1]{Institut f\"ur Mathematik, Universit\"at Potsdam, Am Neuen Palais 10 \break 14469 Potsdam, Germany.\\
\printead{e1}}
\address[a]{Statistical Laboratory, Centre of mathematical sciences, Wilberforce Road, \break Cambridge, CB3 0WB, UK.\\
\printead{e2}}

\runauthor{Carpentier and Kim}

\affiliation{University of Cambridge and University of Cambridge}

\end{aug}

\begin{abstract}
We consider the problem of low rank matrix recovery in a stochastically noisy high dimensional setting. We propose a new estimator for the low rank matrix, based on the iterative hard thresholding method, and that is computationally efficient and simple. We prove that our estimator is optimal both in terms of the Frobenius risk, and in terms of the operator norm risk, i.e.~in terms of the entry-wise risk uniformly over any change of orthonormal basis. This result allows us to provide the limiting distribution of the estimator. In the case where the design is Gaussian, we prove that the entry-wise bias  of the limiting distribution of the estimator is small, which is of great interest for constructing tests and confidence sets for low dimensional subsets of entries of the low rank matrix.
\end{abstract}
\end{frontmatter}

\section{Introduction}

High-dimensional data have generated a great challenge in different fields of statistics, computer science, and machine learning.
In order to consider cases where the number of covariates is larger than the sample size, new methodologies, applicable for the model under some structural constraints, have been developed.
For instance, there have been substantial works under the sparsity assumption including sparse linear regression, sparse covariance matrices estimation or sparse inverse covariance matrices estimation \citep[see e.g.][]{meinshausen2006, bickel2009, huang2008, friedman2008, cai2012}. In this paper, we focus on the problem of \textit{low rank matrix recovery and uncertainty quantification}.

There have been quite a few work on estimating low rank matrices in the matrix regression setting (also named the trace regression setting, the matrix compressed sensing setting, or the quantum tomography setting when the parameter is a density matrix). Many authors \citep[e.g.][]{candes2009, candes2010, recht2011, gross2011} considered the exact recovery of a low-rank matrix based on
a subset of uniformly sampled entries. Also \cite{recht2011, candes2011, FGLE12, GLFBE10, liu} considered matrix recovery based on a small number of noisy linear measurements in the framework of Restricted Isometry Property (RIP). 
\cite{negahban2011} proved non-asymptotic bounds on the Frobenius risk, and investigated
matrix completion under a row/column weighted random sampling.
 \cite{koltchinskii2011} proposed a nuclear norm minimisation method and derived
a general sharp oracle inequality under the condition of restricted isometry property.
Very recently, \cite{cai2015} considered a rank-one projection model and used constrained nuclear norm minimization method to estimate the matrix. 
\cite{FGLE12, GLFBE10} considered a specific quantum tomography problem where the parameter is a density matrix (for more details, plasase see Subsection \ref{ss:q}), and \cite{liu} proved that the quantum tomography design setting satisfies the RIP. In addition, \cite{kol} proposed an estimator based on an entropy minimisation for solving a quantum tomography problem.

\cite{goldfarb2011, tanner2012} adapt the iterative hard thresholding method \citep[first introduced in the sparse linear regression setting, see e.g.~][]{needell2009,blumensath2009} to the problem of low rank matrix recovery in the case where the noise is non-stochastic and of small $L_2$ norm. This procedure has the advantage of being very computationally efficient. In the same vein but applied to the more challenging stochastically noisy setting, \cite{agarwal2012} introduced a soft thresholding technique that provides efficient result in this setting in Frobenius norm, see also~\cite{bunea2011optimal,chen2015fast,klopp2015matrix} for other thresholding methods in related settings that provide results also in Frobenius norm.

Another important problem is on understanding the uncertainty associated to these statistical methodologies, by e.g.~characterizing the limiting distribution of the efficient estimators. Yet results in this area for high dimensional models are still scarce, available mainly for the sparse (generalised) linear regression models \citep{zhang,javanmard, vandegeer2014, nickl}. In the papers~\citep{zhang,javanmard, vandegeer2014}, the authors focus first on constructing an estimator for the sparse parameter that has good properties in $L_{\infty}$ risk, and they use then this result to exhibit the limiting distribution of their estimator. Knowing this limiting distribution immediately enables the construction of tests and confidence sets for low dimensional subsets of parameters.

A similar achievement, i.e.~the construction of an estimator that has an explicit limiting distribution, does not exist in the low rank matrix recovery setting. To the best of our knowledge, moreover, all the theoretical results from the above papers on the estimation of the parameter in the noisy setting are derived in Frobenius risk---neither in the entrywise matrix $L_\infty$ risk, nor in the operator norm (i.e.~the largest singular value). 


In our paper, we consider the problem of constructing an estimator for low-rank matrix in a stochastically noisy high-dimensional setting, under the assumption that a RIP-type isometry condition is satisfied (see Assumption~\ref{ass:designbis}). 
We provide (in~Theorem \ref{th:mainthm2}) error bounds for our estimator in all $p$ Schatten norms for $p>0$. We prove in particular that this estimator has optimal Frobenius and operator norm risk by proving that this estimator has optimal $L_\infty$ risk performance uniformly over any change of orthonormal basis. In addition, a slight modification of our estimator has an explicit Gaussian limiting distribution with bounded bias in operator norm (see~Theorem \ref{thm:asymnorm2}); and in the particular case when the design consists in uncorrelated Gaussian entries, we prove that the bias in $L_{\infty}$ entry wise norm is bounded as well, which is immediately useful for testing hypotheses and constructing confidence intervals for each parameter of interest, similar to the ideas in~\cite{zhang,javanmard, vandegeer2014}. Moreover our estimator is computationally efficient with an explicit algorithm. The proposed  algorithm is inspired by the iterative hard thresholding, that refines its estimation of the matrix by iteratively estimating the low rank sub-space where the matrix's image is defined. It requires only $O(\log n)$ iteration steps to converge approximately, and the computational complexity of the method is of order $O(n d^2 \log n)$ where $d$ is the dimension of the matrix, and $n$ is the sample size. 

In the experiment section we first provide some simulations where we illustrate the efficiency of our method and explain how it can be used to create a confidence interval for the entries of the low rank matrix. We then apply our method to a specific \textit{quantum tomography application}, namely multiple ion tomography~\citep[see, e.g.][]{Guta, GLFBE10,butucea2015spectral, Blatt, acharya2015efficient, holevo2001statistical,nielsen}, where the assumptions required by our method are naturally satisfied~\citep[see e.g.][]{liu, FGLE12}. Finally we compare our method with other existing estimation methods for the trace regression setting~\citep{candes2010, GLFBE10, koltchinskii2011,FGLE12} using the gradient descent implementation of~\cite{agarwal2012} and  also regularized  maximum likelihood based procedures~\citep{butucea2015spectral, acharya2015efficient}. 


As a complement in the Supplementary Material, we adapt our method to the setting of sparse linear regressionm and provide an estimator that has an explicit limiting distribution (recovering the results of \cite{zhang,javanmard, vandegeer2014}).

\section{Setting}

\subsection{Preliminary notations}

For $T>0$, $q \in \mathbb{N}$ and $u \in \mathbb C^q$, we write $\lfloor u\rfloor_{T}$ for the hard thresholded version of $u$ at level $T$, i.e.~for the vector $v$ such that $v_i = u_i \mathbf 1 \{|u_i| \geq T\}$ for $i=1, \ldots, q$. For $q>0$ and $u \in \mathbb R^q$, we write $\|u\|_2 = \sqrt{\sum_{i \leq q} |u_i|^2}$ for the standard $L_2$ norm of $u$, and $\|u\|_{\infty} = \sup_i |u_i|$ for the standard $L_{\infty}$ norm of $u$.

For a $q \times q$ complex matrix $A$, we write $A^T$ as the conjugate transpose of~$A$. We write $\mathrm{tr}(A) = \sum_{k}A_{k,k}$ for the trace of $A$, and $\text{diag}(A)$ for the matrix whose diagonal entries are the same as $A$ while its non-diagonal entries are all zeros. We write the entry-wise matrix norm of $A$ as 
$\|A\|_\infty = \max_{i,j} |A_{i,j}|$, and its squared Frobenius norm as 
$\|A\|_2^2 = \sum_{i,j} A_{i,j}^2$.
We write also the operator norm of $A$ as
$\|A\|_{S} = \max_{i} \lambda_i$, where the $\lambda_i$ are the singular values of $A$, and the Schatten $p$ norm of $A$ for $p > 1$ as $\|A\|_{S_p} = \Big(\sum_{i} \lambda_i^p\Big)^{1/p}$ - and note that $\|A\|_{S_2} = \|A\|_{2}$.

For $T>0$, we write $\lfloor A\rfloor_{T}$ for the hard thresholded version of $A$ at level $T$ for each entry, i.e.~for the matrix $V$ such that $V_{i,j} = A_{i,j} \mathbf 1 \{|A_{i,j}| \geq T\}$ for $i,j=1,\ldots,q$.

\subsection{Model}

Let $d,n \in \mathbb{N}$. Let $\mathcal M$ be the set of $d \times d$ matrices, and 
$$\MM(k),$$
be the set of $d \times d$ complex matrices of rank less than or equal to $k$. Let us also write
$$\MM_{\Omega},$$
for the set of orthonormal matrices in $\MM$.

For $X^i \in \MM, \Theta \in \mathcal M$, we consider the matrix regression problem  where for any $i \leq n$,
\begin{equation*}
Y_i=\mathrm{tr}\big((X^i)^T \Theta\big) + \epsilon_i,
\end{equation*}
where $\epsilon \in \mathbb R^n$ is an i.i.d.~vector of Gaussian white noise, i.e.~$\epsilon \sim \mathcal N(0, I_n)$ (but our results hold in the same way for any sub-Gaussian independent noise $\epsilon$: see Remark~\ref{rem:noise}), and $d \leq n$ but $d^2 \gg n$. Let us write $\mathbb X$ for the linear operator going from $\MM$ to $\mathbb R^n$, and such that for any $A \in \MM$,
$$\mathbb X(A) = \Big( \mathrm{tr}\big((X^i)^T A \big) \Big)_{i \leq n}.$$
The model can be rewritten as
$$Y = \mathbb X(\Theta) + \epsilon,$$
where $Y = (Y_i)_{i \leq n}$. This matrix regression model is directly related to the quantum tomography model (in which case the design $\mathbb X$ is often chosen to be the random Pauli design~\citep{FGLE12, GLFBE10, liu, gross2011,kol}, but it is also  related to e.g.~matrix completion~\citep{negahban2011,kol}.

We state the following assumption on the design operator $\mathbb X$.
\begin{ass}\label{ass:designbis}
Let $K \leq d$. For any $k\leq 2K$, it holds that
\begin{align*}
\sup_{A \in \MM(k)}\Big|  \frac{1}{n}\|\mathbb X(A)\|_2^2 - \|A\|_2^2 \Big| \leq \tilde c_n(k)\|A\|_2^2,
\end{align*}
where $\tilde c_n(k)>0$.
\end{ass}
\begin{remark}\label{rem:gaussian}
The above assumption is very related to the Restricted Isometry Property. Typically, for uncorrelated Gaussian design with mean $0$ and variance $1$ entries, it will hold with probability larger than $1-\delta$ for $\tilde c_n(k) \leq C \sqrt{kd\log(1/\delta)/n}$ where $C>0$ is a universal constant. For the Pauli design used in quantum tomography, it will hold with probability larger than $1-\delta$ for $\tilde c_n(k) \leq C \sqrt{kd\log(d/\delta)/n}$ where $C>0$ is a universal constant~\citep{liu} - see Subsection~\ref{ss:q} for a description of description of a quantum tomography setting in which the Pauli matrices represent measurements. 
\end{remark}

\section{Main results}
As a generalization of sparsity constraints in linear regression models, we impose a rank $k \leq d$ constraint on a matrix $\Theta \in \rr^{d \times d}$. That is, we require the rows (or columns) of $\Theta$ lie in some $k$-dimensional subspace of $\rr^d$. This type of rank constraint arises in numerous applications such as quantum tomography, matrix completion, and matrix compressed sensing \citep[see e.g.][]{FGLE12, GLFBE10, liu, gross2011,negahban2011,koltchinskii2011}.

\subsection{Method}

Our method considers the parameters $B>0, \delta>0, K>0$. The parameter $\delta$ is a small probability that will calibrate the precision of the estimate: the theoretical results that we will prove later for this estimate will hold with probability $1-\delta$, and the smaller $\delta$, the larger the constant in the bound (see Theorem~\ref{th:mainthm2}). The parameter $K$ is an upper bound on two times the actual low rank of the parameter $\Theta$. It does not need to be tight, and the final results will not  depend on it as long $\sqrt{K}\tilde c_n(K) \ll 1$ (see Assumption~\ref{ass:designbis} and Theorem~\ref{th:mainthm2}). The parameter $B$ is an upper bound on the Frobenius norm of the parameter $\Theta$. It again does not need to be tight, but constants in the proof will scale with it.

We set the initial values for the estimator $\hat \Theta^0$ and the threshold $T_0$ such that
$$\hat \Theta^0 =0 \in \rr^{d\times d}, \ \ \ T_0=B \in \rr^+.$$

We update the thresholds $$T_r = 4 \tilde c_n(2K) \sqrt{K} T_{r-1} + \upsilon_n := \rho T_{r-1} + \upsilon_n.$$
where $\upsilon_n = C\sqrt{d\frac{\log(1/\delta)}{n}}$, $C$ is an universal constant (see Lemma \ref{helo}) and $\rho := 4 \tilde c_n(2K) \sqrt{K}$.

Set now recursively, for $r \in \mathbb N$, $r\geq 1$,
$$\hat \Psi^r =\frac{1}{n} \sum_{i=1}^n (X^i)^T\big(Y_i - \mathrm{tr}(X^i\hat \Theta^{r-1})\big) \in \rr^{d\times d},$$
and let $U^{r},V^{r} \in \MM_{\Omega}^2$ be two orthonormal matrices that diagonalise $\hat \Theta^{r-1} + \hat \Psi^r$.
Then we set
\begin{equation}\label{eq:estimator2}
\hat \Theta^r =  U^{r} \lfloor (U^{r})^T (\hat \Theta^{r-1} +\hat \Psi^r ) V^{r} \rfloor_{T_r} (V^{r})^T.
\end{equation}
This procedure provides a sequence of estimates, and as we will prove in the next subsection, this sequence is with high probability close to the true $\Theta$ as soon as $r$ is of order $\log(n)$ (see Theorems~\ref{th:mainthm2} and~\ref{thm:asymnorm2}).

\begin{remark}\label{rem:para}
Note that although we describe this method using many quantities, 
in fact while implementing our method we only need to set up four quantities: $\rho, \upsilon_n, T_0$ and the stopping time $r$. We describe in Equation~\eqref{eq:sr} how to implement a good stopping rule, and in Subsection~\ref{ss:disc} how to choose the three first parameters. In particular, $T_0$ can be chosen in a data driven way.
\end{remark}

This method is related to Iterative Hard Thresholding (IHT), a method that has been developed for the sparse regression setting \citep[see e.g.][]{blumensath2009, needell2009}. It is less straightforward to see this in this setting, as in the sparse regression setting where we adapt also our method in Subsection~\ref{ss:sr}, and for a more comprehensive discussion of the relation between our method and IHT, see the Remark~\ref{rem:iht}. Note that IHT algorithms have been proved to work in settings where the noise is small and non-stochastic \citep[see e.g.][]{blumensath2009, needell2009,goldfarb2011, tanner2012}, but to the best of our knowledge, there are no results on IHT in a stochastically noisy setting.


\subsection{Results for the low rank matrix recovery}

\paragraph{Main result for our thresholded estimator} We now provide a theorem that guarantees that the estimate $\hat \Theta^r$ after $O(\log(n))$ iterations has at most rank $k$, and its entry-wise $L_\infty$ risk and Frobenius risk are bounded with the optimal rates (see the first point in Subsection 3.3).
\begin{thm}\label{th:mainthm2}
Assume that Assumption~\ref{ass:designbis} is satisfied and that $\tilde c_n(2K) \sqrt{K} < 1/4$. Let $r \approx O(\log(n))$. We have that for a constant $C_1>0$ it holds that with probability larger than $1-\delta$ and for any $k \leq K/2$
\begin{equation*}
\sup_{\Theta \in \MM(k), \|\Theta\|_{2} \leq B} \| \Theta  - \hat \Theta^r\|_S \leq  C_1 \sqrt{\frac{d\log(1/\delta) }{n}},
\end{equation*}
and also that
$$\sup_{\Theta \in \MM(k), \|\Theta\|_{2} \leq B} \mathrm{rank}(\hat \Theta^r) \leq k,$$
and also that for any $p >0$
$$\sup_{\Theta \in \MM(k), \|\Theta\|_{2} \leq B} \| \Theta - \hat \Theta^r\|_{S_p} \leq  C_1 k^{1/p} \sqrt{\frac{d\log(1/\delta) }{n}}.$$
\end{thm}
The above theorem proves among other things that our estimate attains the minimax optimal Schatten $p$ risk, which other estimates in the literature also attain for e.g.~$p=2$. The first interesting property is that our proposed estimator has an explicit algorithmic form and is very computationally efficient. Another interesting additional property is that it is also minimax-optimal in operator norm (or entry-wise matrix $L_{\infty}$ risk), and that the entry-wise error is not more than $\sqrt{d/n}$ with high probability \textit{for any orthonormal change of basis of the matrix $\Theta$}. This is a strong result since the entry-wise norm is not invariant by orthonormal change of basis while the Frobenius norm is. This result is already useful for measuring the uncertainty of an estimate (in particular since it does not require the a priori knowledge of the rank of the matrix $\Theta$). 

\paragraph{Asymptotic normality results} To prove asymptotic normality, we slightly modify the estimator defined in Theorem \ref{th:mainthm2}. Consider the estimator $\hat \Theta^r$ of Theorem~\ref{th:mainthm2} (with $r \approx O(\log(n))$) and define
$$\hat \Theta = \hat \Theta^r + \frac{1}{n}\sum_{i=1}^n (X^i)^T [Y_i  -  \mathrm{tr}( (X^i)^T\hat \Theta^r)].$$


\begin{thm}\label{thm:asymnorm2}
Set \begin{align*} 
Z &:=\frac{1}{\sqrt{n}} \sum_{i\leq n} (X^i)^T \epsilon_i \\
\Delta &:= \sqrt{n}(\hat \Theta^r - \Theta)  - \frac{1}{\sqrt{n}}\sum_{i\leq n} (X^i)^T\mathrm{tr}\big((X^i)^T(\hat \Theta^r - \Theta)\big).
\end{align*} 
Then we have
\begin{equation}\label{eq:split2}
\sqrt{n}(\hat \Theta - \Theta) = \Delta +Z,
\end{equation}
where $Z|\mathbb X \sim \mathcal N \Big(0,  \big(\frac{1}{n}\sum_{i \leq n} (X^i_{j,j'}X^i_{l,l'})\big)_{j,j',l,l'}\Big)$. 

Let $r \approx O(\log(n))$. The two following bounds hold for the bias term $\Delta$ under two different assumptions.
\begin{itemize}
\item Assume that Assumption~\ref{ass:designbis} is satisfied for some $K>0$ and that $\tilde c_n(2K) \sqrt{K} = o(1)$. If the rank of $\Theta$ is smaller than $2K$ and if its Frobenius norm is bounded by $B$, there is a constant $C_1>0$ such that with probability larger than $1-\delta$ 
$$\frac{\| \Delta \|_S}{\sqrt{d}} \leq 4C_1 \tilde c_n(2K) \sqrt{K}\log(1/\delta) = o_{\mathbb{P}} (1).$$
\item Assume that the elements in the design matrices $X^i \in \MM$ are i.i.d.~Gaussian with mean $0$ and variance $1$, and that $\max(K^2d, Kd\log(d)) = o(n)$, we have that
$$\| \Delta \|_\infty = o_{\mathbb{P}} (1).$$
Note that this implies the previous result.
\end{itemize}


\end{thm}

This theorem, which is in the spirit of the works in the context of sparse linear regression of~\cite{zhang,javanmard, vandegeer2014}, implies that there exists an estimator of $\Theta$ that has a Gaussian limiting distribution, and whose rescaled bias $\Delta$ with respect to $\Theta$ can be bounded $i)$ in operator norm under our Assumption~\ref{ass:designbis} and $ii)$ in $L_{\infty}$ norm as well in the specific case where the design is Gaussian.


\begin{remark}\label{rem:noise}
Theorems~\ref{th:mainthm2} and~\ref{thm:asymnorm2} are proved for a Gaussian noise $\epsilon$, but these results are easily generalisable to any independent, sub-Gaussian noise, with a similar but more technical proof (based on Talagrand's inequality). The results of Theorem~\ref{thm:asymnorm2} would however be modified in that the random variable $Z$, conditioned on the design $\mathbb X$, would then not be exactly Gaussian, but have a limiting Gaussian distribution using the central limit theorem.
\end{remark}


\paragraph{Stopping rule $r$} Theorem~\ref{th:mainthm2} is satisfied after $r= O(\log(n))$ iterations of our thresholding strategy, and so we know what is a theoretical value for $r$ so that our strategy works. However, it is possible to propose a data driven stopping rule that will perform well. For a desired precision $e>0$, we propose to stop the algorithm as soon (after having thresholded a last time) as
\begin{equation}\label{eq:sr}
T_r \leq (1+e) \frac{1}{1-\rho} v_n.
\end{equation}
Let us write $\hat r$ for the time where the stopping rule stops. The following result holds for the estimator stopped at this stopping rule.

\begin{thm}\label{prop:srule}
Assume that Assumption~\ref{ass:designbis} is satisfied and that $\tilde c_n(2K) \sqrt{K} < 1/8$, i.e.~$\rho \leq 1/2$. Let $e \leq 0.1$ in (\ref{eq:sr}. The estimator $\hat \Theta^{\hat r}$ satisfies with probability larger than $1-\delta$ and for any $k \leq K/2$
\begin{equation*}
\sup_{\Theta \in \MM(k), \|\Theta\|_{2} \leq B} \| \Theta  - \hat \Theta^{\hat r}\|_{S} \leq  \frac{1.1}{1-\rho} v_n = 2.2 C \sqrt{\frac{d\log(1/\delta) }{n}},
\end{equation*}
and also that
$$\sup_{\Theta \in \MM(k), \|\Theta\|_{2} \leq B} \mathrm{rank}(\hat \Theta^{\hat r}) \leq k,$$
which implies for any $p>0$
$$\sup_{\Theta \in \MM(k), \|\Theta\|_{2} \leq B} \| \Theta - \hat \Theta^{\hat r}\|_{S_p} \leq  2.2C (2k)^{1/p} \sqrt{\frac{d\log(1/\delta) }{n}}.$$
Moreover $\hat r$ is such that
$$\hat r \leq 1 + \frac{\log\Big(10(1-\rho)T_0/(v_n)\Big)}{\log(1/\rho)} \leq O(\log(n)).$$
\end{thm}
This empirical stopping rule, that does not require the tuning of any additional parameters\footnote{In Subsection \ref{s:sim} we introduced the desired precision $e$. The precision $e$ is a very natural quantity to choose for the experimenter---one can set e.g.~$e=0.1$ for an $1.1$ optimal solution with respect to the solution outputted by an algorithm that runs for an infinitely long time}, is guaranteeing minimax optimal results in less than $\log(n)$ iterations. Note that Theorem~\ref{thm:asymnorm2} would also hold using this stopping rule - this can be proved in the same way as Theorem~\ref{prop:srule} is proved. 

\subsection{Discussion}\label{ss:disc}

\paragraph{Comparison of our results with the literature} Our Theorem~\ref{th:mainthm2} gives bounds for our estimators in all Schatten $p>0$ norms (including the operator norm, and therefore uniform entry wise bounds in all rotation basis). A first point is that our results are minimax optimal in both Frobenius and operator norm. The corresponding lower bound in Frobenius norm can be found in e.g.~ Theorem 5 of \cite{koltchinskii2011} (under an assumption related to our Assumption~\ref{ass:designbis}) or~\cite{candes2011} under an assumption that is the same as ours. The corresponding lower bound in operator norm can be found in e.g.~\cite{carpentier2015uncertainty}. In addition to this, the paper~\cite{koltchinskii2015optimal} even contains further lower bounds results proving that the operator norm rate $\sqrt{d/n}$ (and associated Schatten $q$ norm $k^{1/q}\sqrt{d/n}$) is optimal also in the case of quantum tomography that we use in our experiments later on, i.e.~under the additional assumptions that the parameter is a density matrix and that the design is random Pauli. To the best of our knowledge, our method is the first iterative method that has such an optimality property in operator norm - for instance, the paper~\cite{koltchinskii2015optimal} provides results for Schatten norms with $q \in [1,2]$, but not for other Schatten norms. Besides, we proved in Theorem~\ref{thm:asymnorm2} a slight modification of our estimator has an explicit Gaussian limiting distribution, and this is, again to the best of our knowledge, the first iterative method for low rank matrix recovery that has such a property. On top of that, the computational complexity of our algorithm is low as for any procedure based on iterative hard thresholding : see the papers~\citep{goldfarb2011, tanner2012}. Our assumption~\ref{ass:designbis} is a strong RIP condition. But it is in particular satisfied in the interesting application of multiple ion tomography for the natural Pauli design as soon as the number of settings is large enough, see Subsection~\ref{ss:q}.

Operator norm bounds are particularly interesting since they provide an entrywise bound up to any change of orthonormal basis. In particular, they provide a bound on the eigen values - and since these bounds do not depend on the true rank $k$, they can be used to implement conservative confidence sets. Moreover as highlighted in the papers~\cite{zhang,javanmard, vandegeer2014}, having a bound on the entrywise risk, and then an estimator with explicit limiting distribution, is interesting in that it can be used to construct tests and confidence intervals for subsets of coordinates of the parameter $\Theta$. We illustrate this point in the Simulation section (see Section~\ref{s:sim}), where a confidence set is constructed using the limiting distribution. Note however that the bound on the bias term $\Delta$ in $L_{\infty}$ norm in Theorem~\ref{thm:asymnorm2} requires the fact that the design is Gaussian. On the other hand, the bound on the bias term $\Delta$ in operator norm in Theorem~\ref{thm:asymnorm2} requires only the fact that our assumption~\ref{ass:designbis} is satisfied.

\paragraph{Stopping rule $r$} Our theorems are satisfied after $r= O(\log(n))$ iterations of our thresholding strategy, and so we know what is a theoretical value for $r$ so that our strategy works. We also defined an empirical stopping rule, see~\eqref{eq:sr} and Theorem~\ref{prop:srule}. We use this stopping rule in practice for all our experiments in Section~\eqref{s:ex}.

\paragraph{Calibration of the parameters of the proposed method} Our method is not parameter free - there are three quantities that need to be calibrated. The two first ones, that we write $\rho$ and $\upsilon_n$, enter in the definition of the thresholds sequence $(T_r)_r$. $\rho$ controls the rate at which we make our threshold decay, and $\upsilon_n/(1-\rho)$ is the quantity toward which it converges when $r$ goes to infinity. The last quantity, namely $T_0$ is the initialisation of the threshold sequences. Here are some comments on how to choose these quantities:
\begin{itemize}
\item {\bf Rate of decay $\rho$ :} In theory the parameter $\rho$ can be taken between $1$ and $4\sqrt{K}\tilde c_n(2K)$ where $K$ is an upper bound on the rank of the parameter and $\tilde c_n(2K)$ is the constant associated to the design such that Assumption~\ref{ass:design} is satisfied. It might not be really possible to compute exactly $K$ or $\tilde c_n(2K)$ without more assumptions on the design. But there is at least one design that is interesting in practice, namely the random Pauli design for quantum tomography, that is such that we have an upper bound on $\tilde c_n(2K)$ for all $K$ that is of order $\sqrt{Kd\log(d)/n}$ with high probability, i.e.~it decays with $n$ (the same holds in Gaussian design up to the $\log$ term). In this design if $n$ is large enough, we know that taking $\rho = 1/2$ will work - we do not want to take $\rho$ too close to $1$ since the bound on the performance of the estimator scales with $1/(1-\rho)$.
\item {\bf Smallest threshold calibration $\upsilon_n$} The interpretation and theoretical value of $\upsilon_n$ is clear: it should be taken to be larger than the $\delta$ quantile of the LHS quantity defined in Equation~\eqref{eq:1} divided by $\|A\|_2$. Now since we do not have access to this quantile, we calibrate it in the experimental section of this paper as an empirical estimator of the asymptotic quantile described above (using Theorem~\ref{thm:asymnorm2}).
\item {\bf Initialisation threshold $T_0$ :} The constant $T_0$ needs to be taken as an upper bound on the Frobenius norm of $\Theta$. Note first that estimating from the data an upper bound on $\|\Theta\|_2^2$ is easy under Assumption~\ref{ass:design}, i.e.~the quantity for $\kappa>0$
$$\frac{1}{n}\|Y\|_2^2(1+\kappa),$$
overestimates $\|\Theta\|_2^2$ . 
In our simulations, we propose a slightly more refined heuristic upper bound and use the same $T_0$ and $T_r$ in Subsection~\ref{s:sim} and \ref{ss:q}.
\end{itemize}
To conclude on the practical tuning of the constants, we would like to emphasize that at least in practical situation, namely quantum tomography, we have enough information about both the design and the noise level to know that the above calibration will be working, provided that the target matrix is indeed low rank. So although the tuning of parameters is always a tricky issue for any algorithm, in at least this specific application, our algorithm can be used as it is. 

\section{Experiments}\label{s:ex}

In this section we present some experiments, first some simulation for the construction of confidence intervals, and then a formal comparison of our thresholded estimator with other methods on specific quantum tomography problem, namely multiple ion tomography.

\subsection{Simulation results for the construction of entry-wise confidence intervals}\label{s:sim}



We performed experiments for low-rank matrix recovery, with matrix dimension $d$ . We consider a Gaussian design where each $X^i_{j,j'} \sim \mathcal N(0,1)$ and are independent. We also consider a Gaussian uncorrelated noise $\epsilon \sim \mathcal N(0, I_n)$. We consider a parameter $\Theta$ of rank $k$ that is stochastically generated in an isotropic way as
$$\Theta = \sum_{l = 1}^k N_l N_l^T,\quad \mathrm{where,} \quad N_l \sim \mathcal N(0, I_d).$$

We implemented our method choosing a data-driven heuristic for the choice of our parameters. We first set
$$\hat \Theta^0 = 0.$$
We set for any $r\geq 1$
\begin{equation}\label{emprisk} \hat \sigma_r^2 = \|Y - (\mathrm{tr}\big((X^i)^T \hat \Theta^{r-1}\big)_{i\leq n}\|_2^2/n,
\end{equation}
i.e.~the empirical risk, and
\begin{equation}\label{upsilonr}
\upsilon_{n}(r) = \hat \sigma_r \sqrt{\frac{d}{n}} q_{90\%},
\end{equation}
where $q_{90\%}$ is the $90\%$ quantile of a $\mathcal N(0,1)$ random variable. $\upsilon_n(r)$ replaces here $\upsilon_n$, and is a heuristic high probability bound on the error for each coordinate.

 We set
\begin{equation}\label{T1}
T_1 = B = \hat \sigma_1 +  \upsilon_{n}(1),
\end{equation}
which is by construction higher than the Frobenius norm of $\Theta$ with high probability, and
\begin{equation}\label{Tr}
T_r = \rho T_{r-1} + \upsilon_{n}(r),
\end{equation}
where we select $\rho = 1/2$ (we take $1/2$ so that the decay is not too fast, but also so that $1/(1-\rho)$ is not too large).

We also use the heuristic stopping rule described in Equation~\eqref{eq:sr}, i.e.~we iterate until
$$T_r \leq (1 + e) \times \frac{1}{1 - \rho} \upsilon_{n}(r) = 2.2 \upsilon_{n}(r),$$
for $e = 0.1$.

We also construct, using the limiting distribution results provided in~Theorem \ref{thm:asymnorm2}, a confidence set for the all the entries of $\Theta$ that is such that for any entry $(m,m')$, we set the confidence interval
$$C_n^{m,m'} = [\hat \theta_{m,m'} - c_{m,m'}, \hat \theta_{m,m'} + c_{m,m'}],$$
where
$$c_{m,m'} = \hat \sigma_r \hat \Sigma_{m,m'} \frac{q_{95\%}}{\sqrt{n}},$$
where $\hat \Sigma_{m,m'}^2 = 1/n \times \sum_{i \leq n} (X_{m,m'}^i)^2$.

We provide several results, depending on the values of $(n,p,k)$, averaged over $100$ iterations of simulations. For these simulations, we present three kinds of results:
\begin{itemize}
\item A first set of graphs (Figure~\ref{fig:1}) presents, for different values of $p,k$, and in function of the sample size $n$, the logarithm of the rescaled Frobenius risk of the estimate $\hat \Theta$, i.e.
$$\log\Big(\frac{\|\hat \Theta - \Theta\|_2}{\|\Theta\|_2}\Big).$$
\item A second set of graphs (Figure~\ref{fig:2}) presents, for different values of $p,k$, and in function of the sample size $n$, the logarithm of the averaged diameter of the confidence intervals $C_n^{m,m'}$, i.e.
$$\log\Big(\frac{1}{d^2} \sum_{m,m'} c_{m,m'}\Big).$$
\item A last set of graphs (Figure~\ref{fig:3}) presents, for different values of $p,k$, and in function of the sample size $n$, the averaged coverage probability of the confidence intervals $C_n^{m,m'}$, i.e.
$$\frac{1}{d^2} \sum_{m,m'} \mathbf 1 \{ \theta_{m,m'} \in C_n^{m,m'}\}.$$
\end{itemize}
All these graphs also exhibit $95\%$ confidence intervals (upper and lower $2.5\%$ quantile values from 100 iterations) around their means (dotted lines in the graphs, the solid line being the mean).


\begin{center}
\begin{figure}
\begin{minipage}{0.45\textwidth}
 \includegraphics[width=1\textwidth]{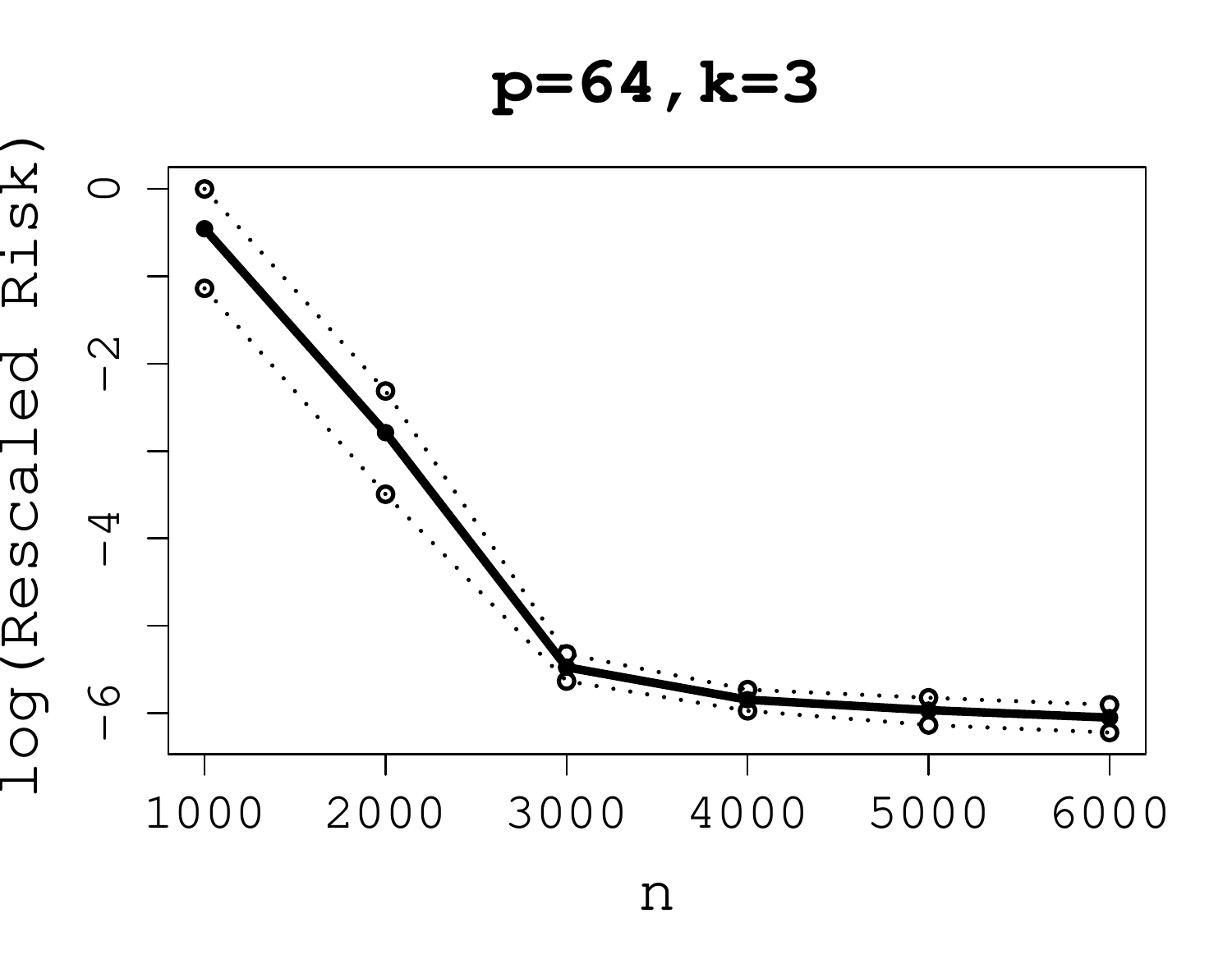}
\end{minipage}
\begin{minipage}{0.45\textwidth}
 \includegraphics[width=1\textwidth]{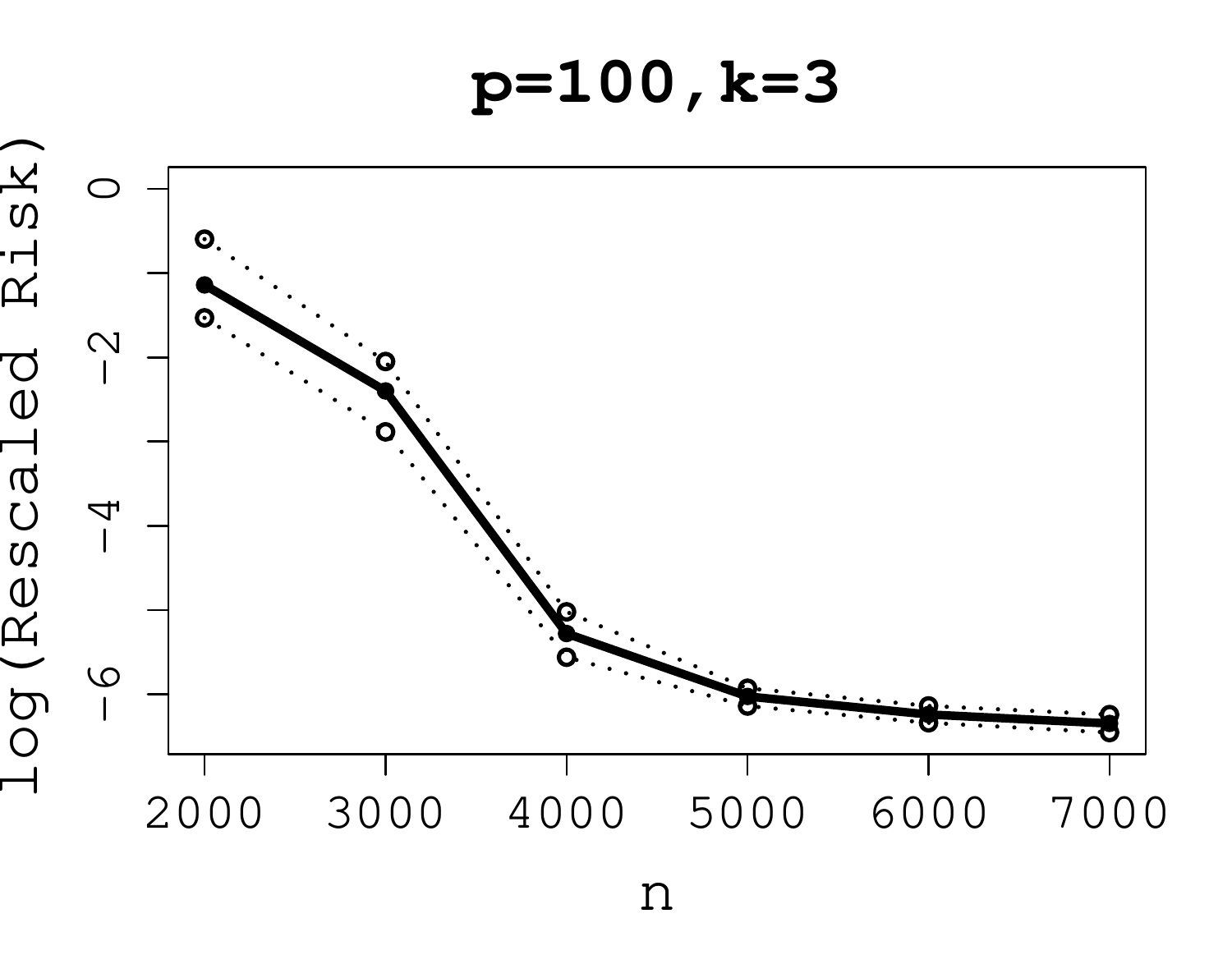}
\end{minipage}

\vspace{-0.3cm}

\begin{minipage}{0.45\textwidth}
 \includegraphics[width=1\textwidth]{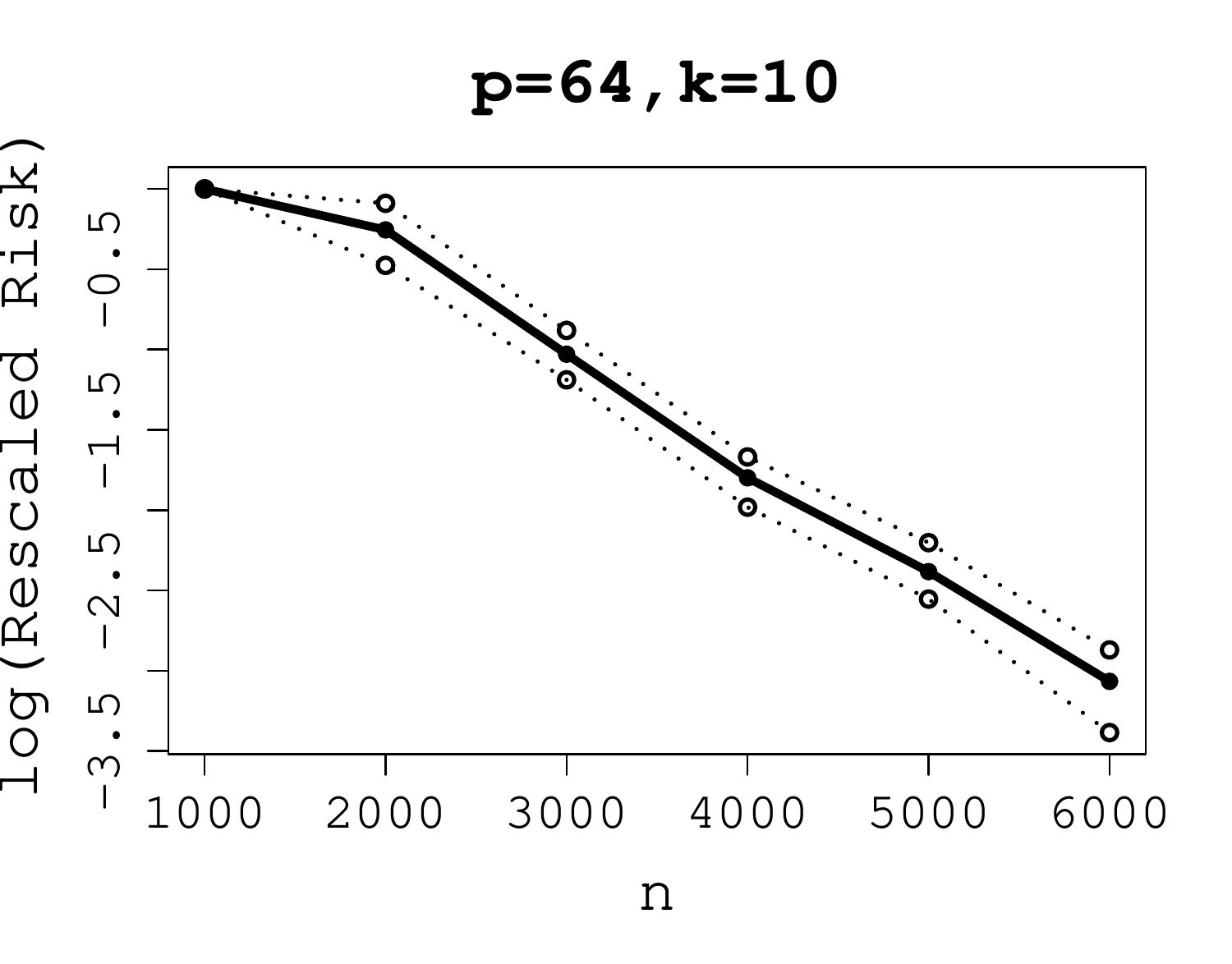}
\end{minipage}
\begin{minipage}{0.45\textwidth}
 \includegraphics[width=1\textwidth]{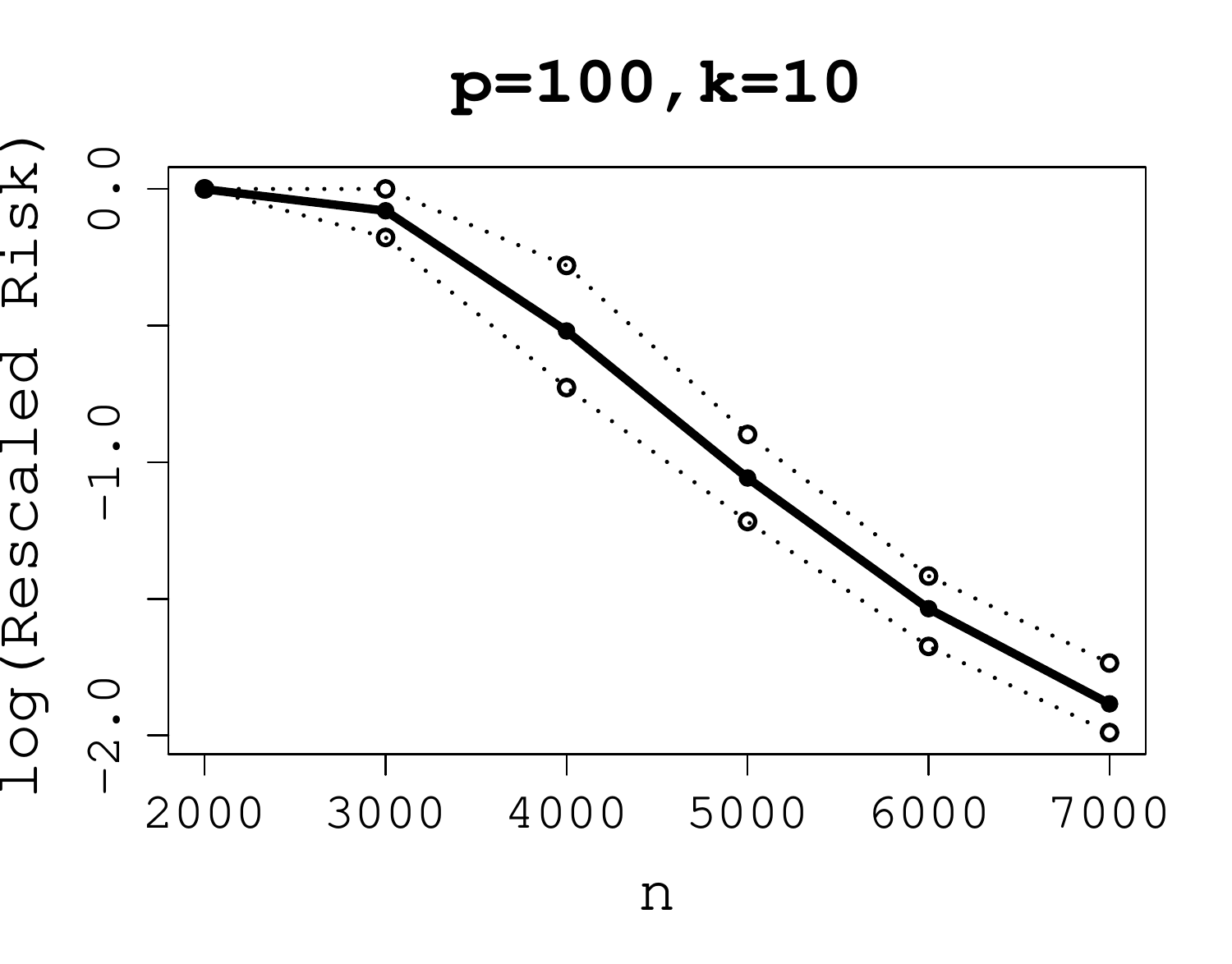}
\end{minipage}
\vspace{-0.2cm}
\caption{Logarithm of the rescaled Frobenius risk of the estimate in function of $n$, for different values of $p,k$. The solid line is the average over $100$ iterations, the dotted lines form $95\%$ confidence intervals.} \label{fig:1}
\end{figure}
\end{center}

\begin{center}
\begin{figure}
\begin{minipage}{0.45\textwidth}
 \includegraphics[width=1\textwidth]{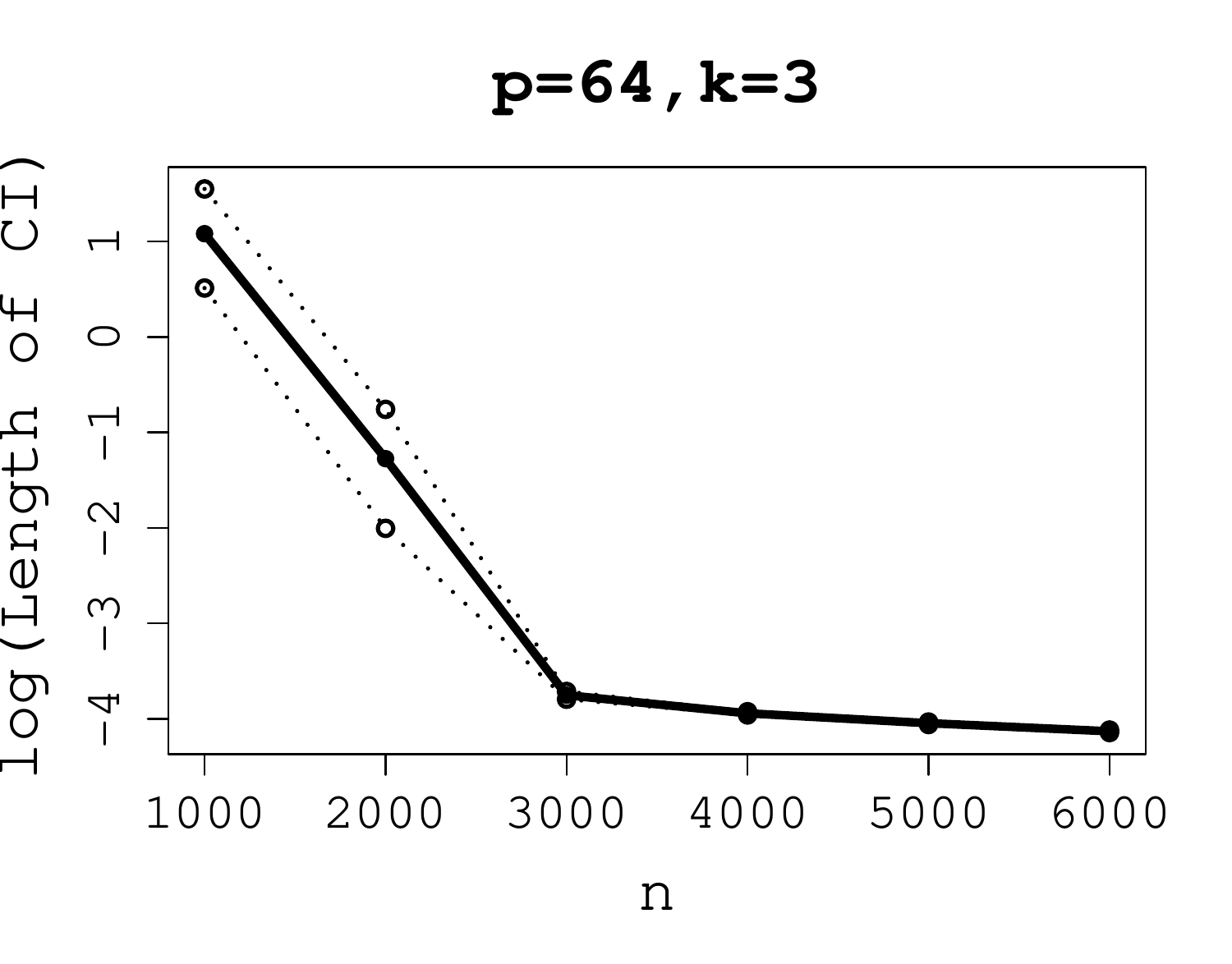}
\end{minipage}
\begin{minipage}{0.45\textwidth}
 \includegraphics[width=1\textwidth]{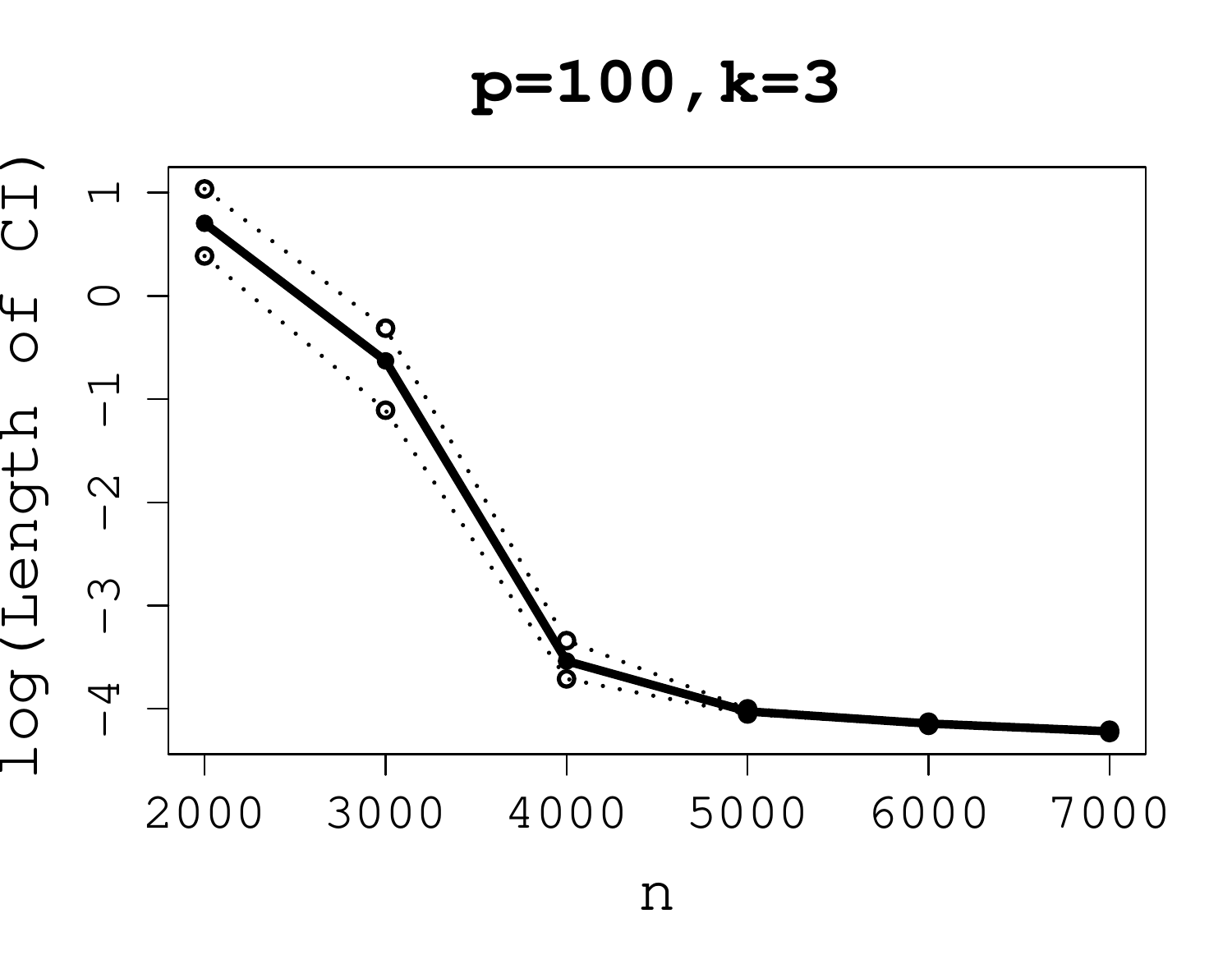}
\end{minipage}

\vspace{-0.3cm}

\begin{minipage}{0.45\textwidth}
 \includegraphics[width=1\textwidth]{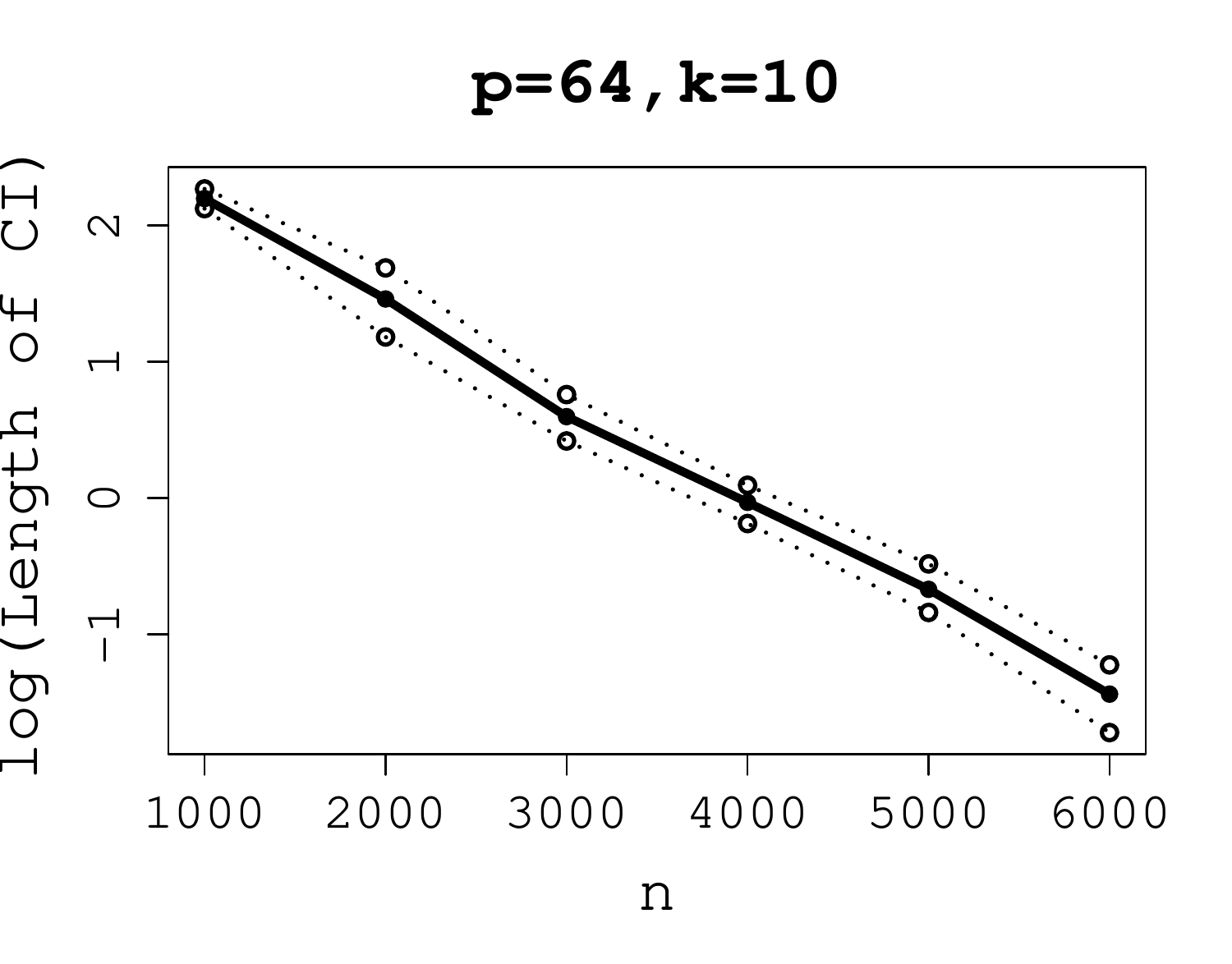}
\end{minipage}
\begin{minipage}{0.45\textwidth}
 \includegraphics[width=1\textwidth]{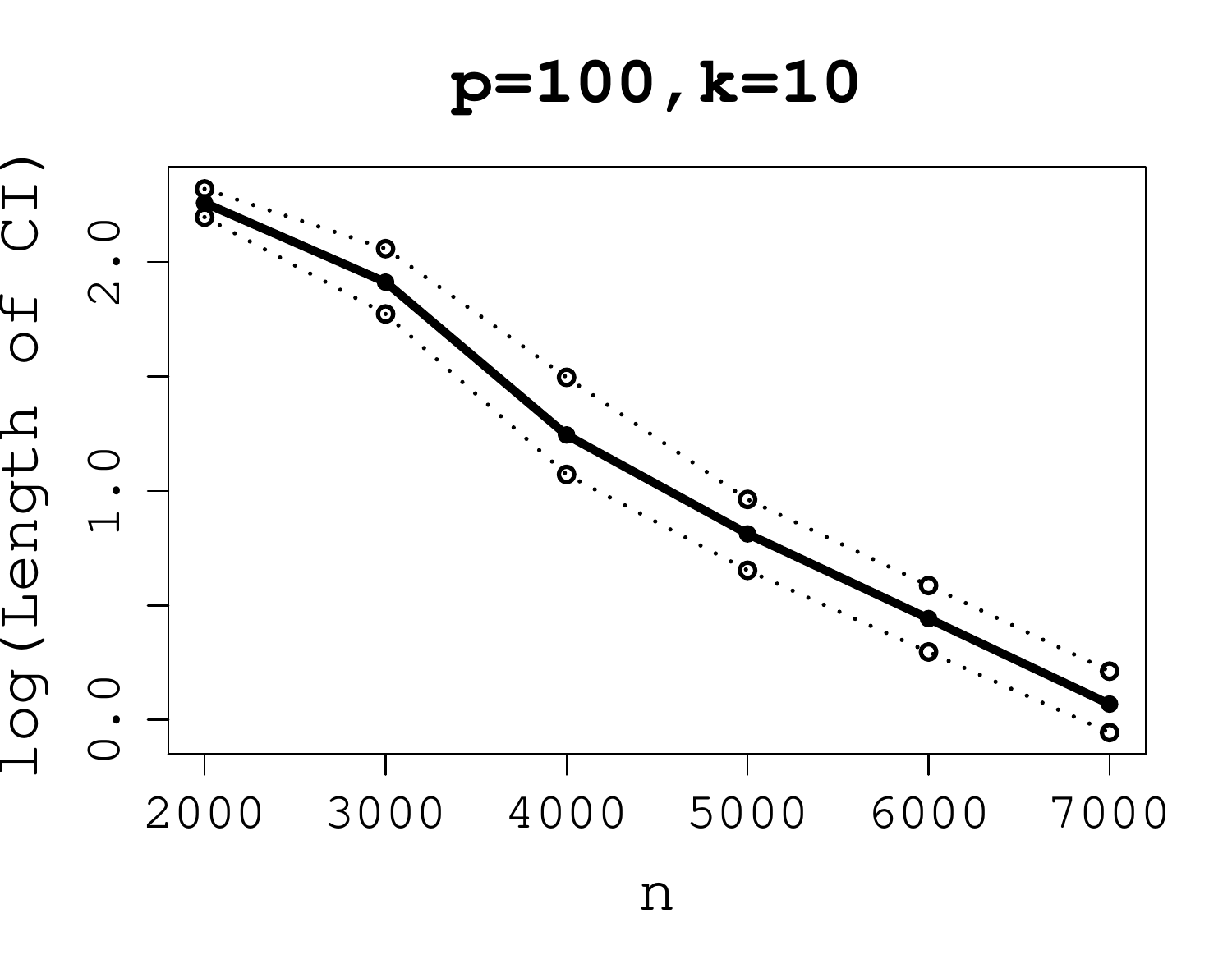}
\end{minipage}
\vspace{-0.2cm}
\caption{Logarithm of the averaged rescaled length of the confidence intervals of the in function of $n$, for different values of $p,k$. The solid line is the average  over $100$ iterations, the dotted lines form $95\%$ confidence intervals.} \label{fig:2}
\end{figure}
\end{center}

\begin{center}
\begin{figure}
\begin{minipage}{0.45\textwidth}
 \includegraphics[width=1\textwidth]{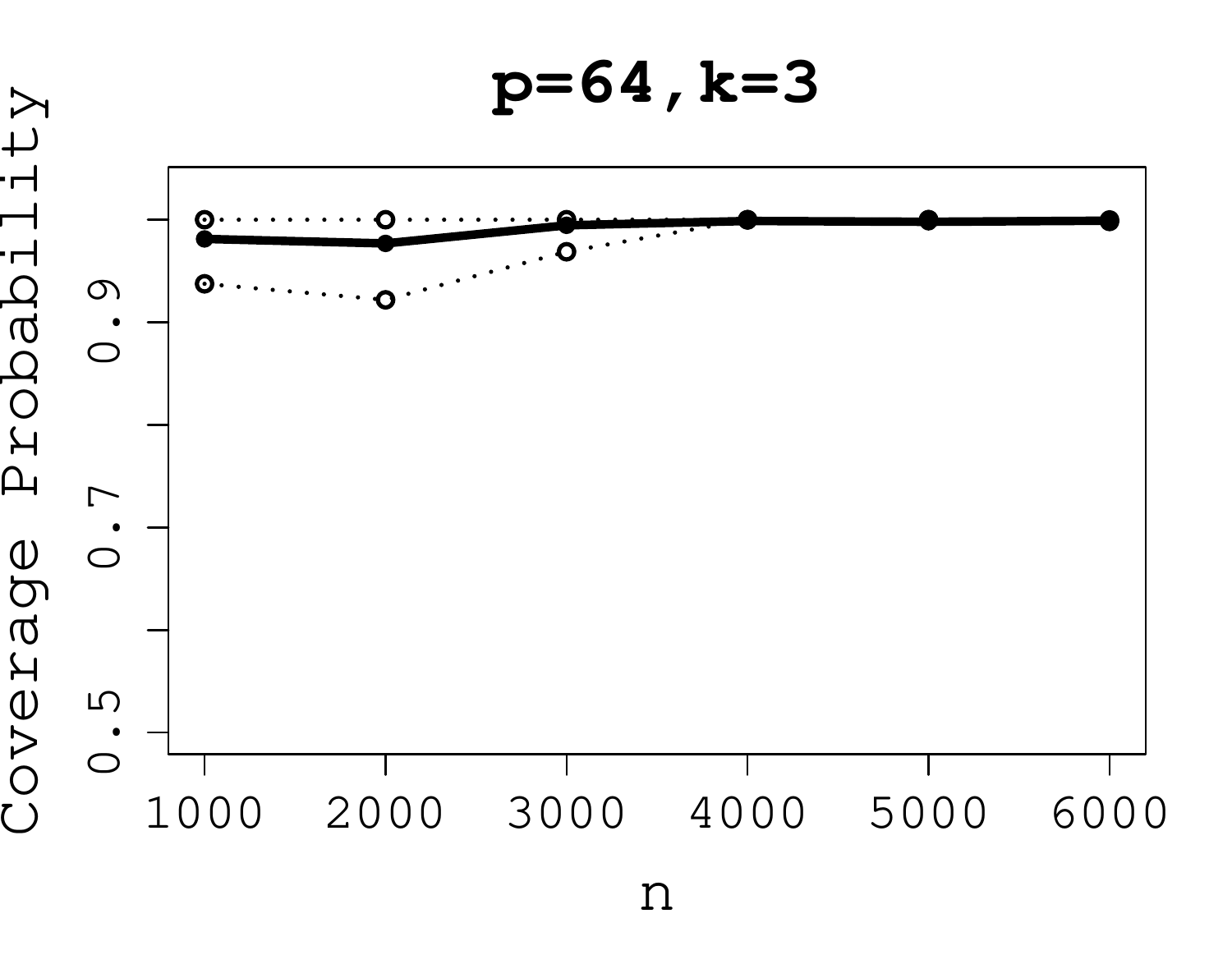}
\end{minipage}
\begin{minipage}{0.45\textwidth}
 \includegraphics[width=1\textwidth]{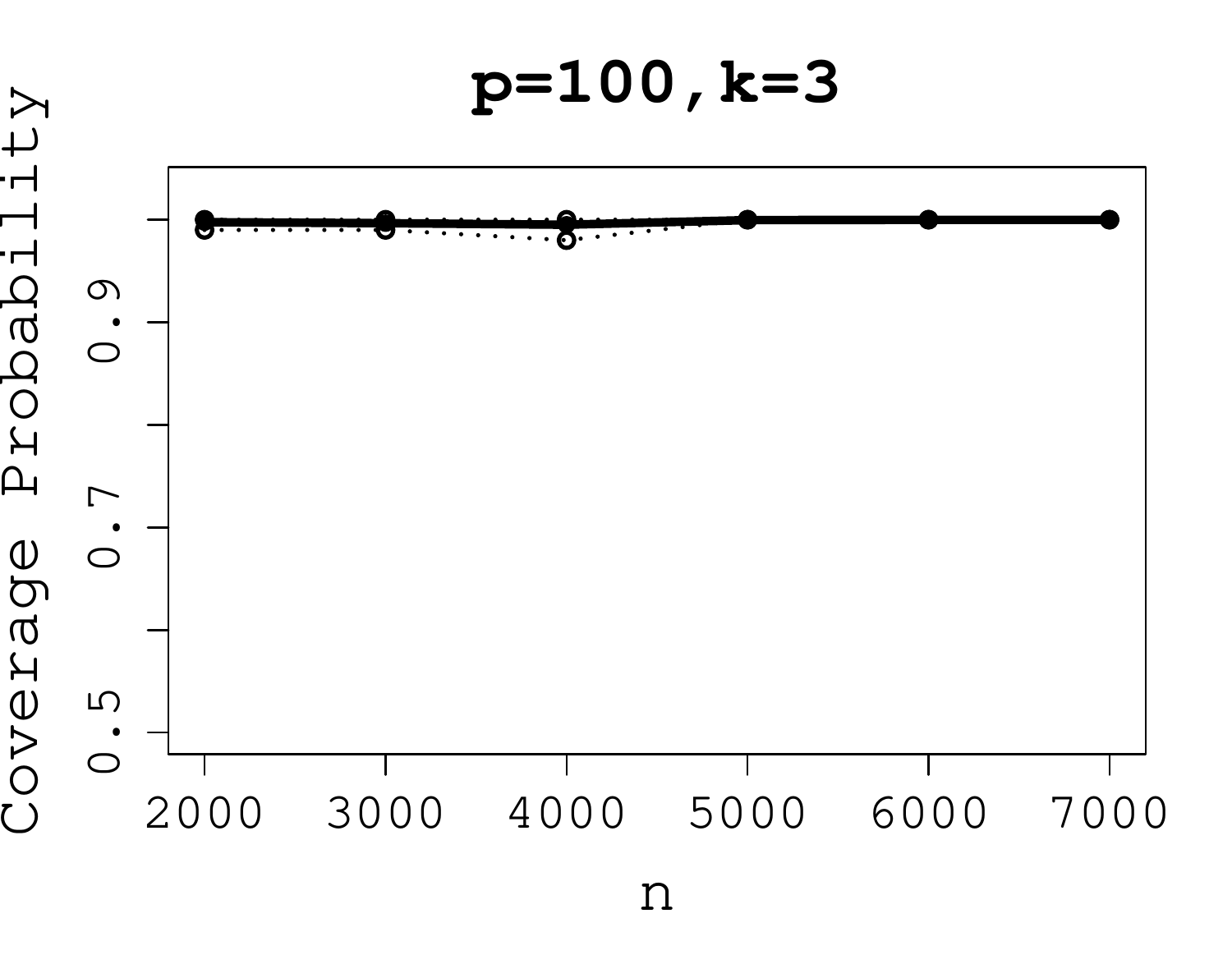}
\end{minipage}

\vspace{-0.3cm}

\begin{minipage}{0.45\textwidth}
 \includegraphics[width=1\textwidth]{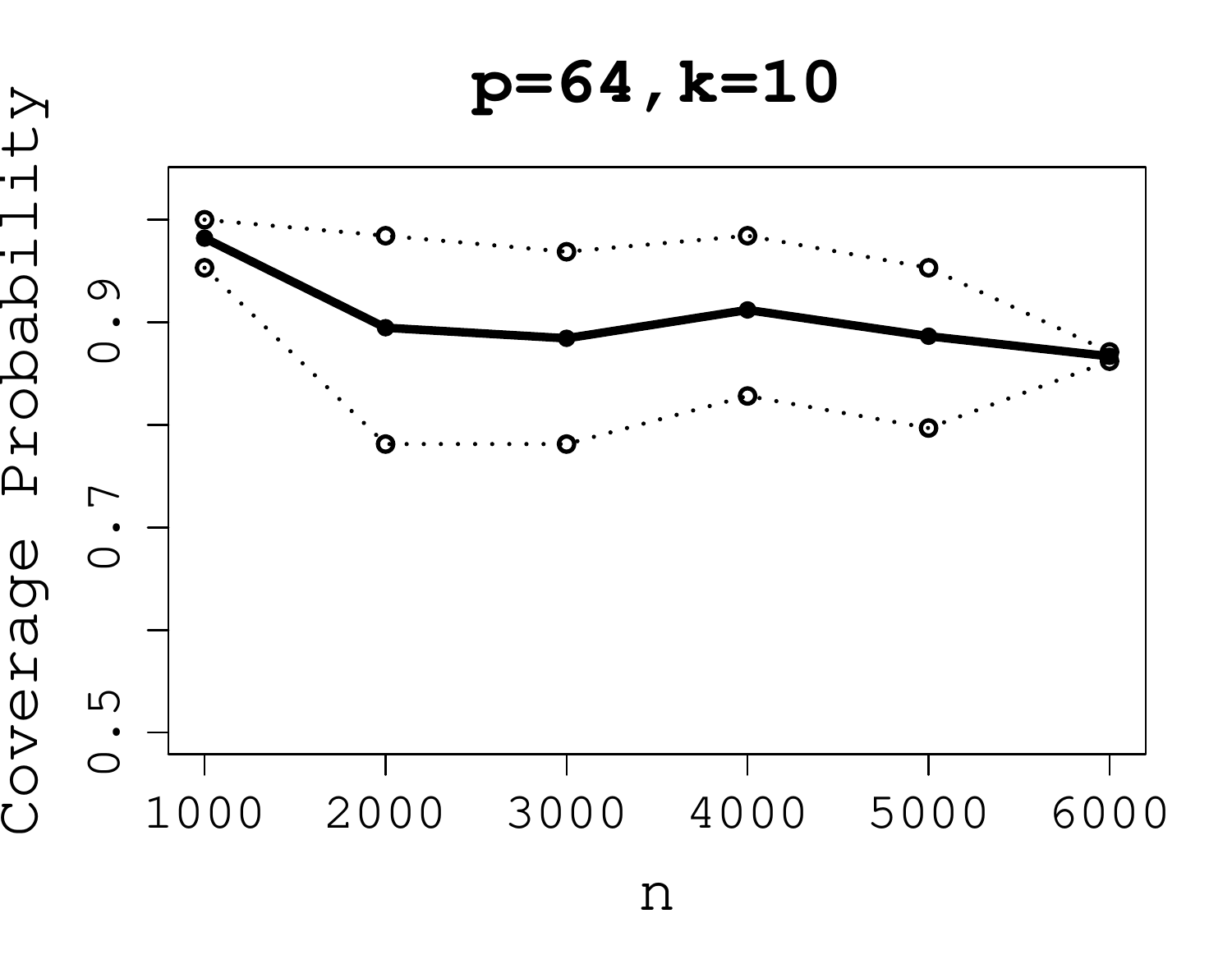}
\end{minipage}
\begin{minipage}{0.45\textwidth}
 \includegraphics[width=1\textwidth]{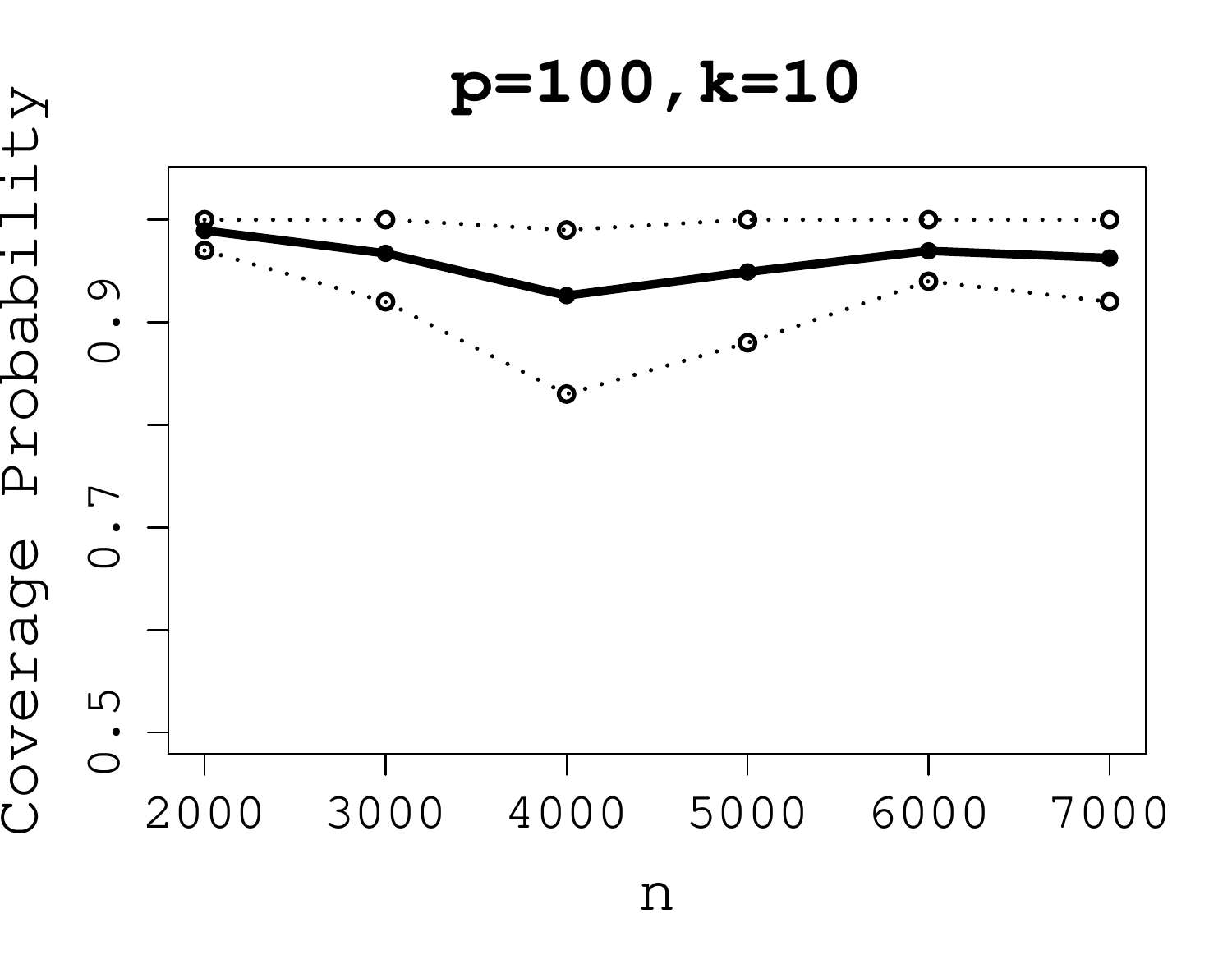}
\end{minipage}
\vspace{-0.2cm}
\caption{Averaged coverage of the confidence intervals of the in function of $n$, for different values of $p,k$. The solid line is the average over $100$ iterations the dotted lines form $95\%$ confidence intervals.} \label{fig:3}
\end{figure}
\end{center}

These figures exhibit different behaviours depending on the difficulty of the problems (increasing with $p$ and more importantly with $k$). The graphs in Figure~\ref{fig:1} for $k=3$ (and $p \in \{64, 100\}$) exhibit first a very fast decay of the risk, until some critical threshold $n = ckd$ where $c$ seems to be between $10$ and $20$. At this point, one can actually observe that the method recovers in most case the true rank $k$ of the matrix, whereas it before recovered only a smaller rank approximate of $\Theta$---with a too small $n$, it could not distinguish all the signal from the noise, and the fact that it gradually does for larger $n$ explains the fast decay of the logarithm of the rescaled risk. After that, the curve has a kink and the decay becomes slower (the theory predicts that the logarithm of the rescaled risk should decrease with $n$ as $-\log(n)$). After this kink, all the $k$ ``rank directions" have been identified, and the logarithm of the rescaled risk starts decreasing slower, according to the theoretical rate of $-\log(n)$. The graphs in Figure~\ref{fig:1} for $k=10$ (and $p \in \{64, 100\}$)  exhibit mainly the first regime, since $k$ is larger and the second regimes comes for larger values of $n$---empirically, we can observe that the method recovers most of the time all $k$ ``directions" as soon as $n = 4000$ for $p = 64$, as soon as $n = 6000$ for $p = 100$. 

A parallel evolution can be observed in Figure~\ref{fig:2}, for the logarithm of the average length of the confidence intervals. It is not at all surprising since this length is supposed to reflect the risk. The averaged coverage of these intervals in Figure~\ref{fig:3} is in average larger than $87\%$ in all cases, and in more than $95\%$ of the cases, it is higher than $74\%$ in all cases, which makes our method reliable.


\subsection{Quantum tomography experiments}\label{ss:q}

\subsubsection{Description of the ion tomography setting}

An important application that satisfies our assumptions and to which our method can be applied is \textit{quantum tomography}, i.e.~the estimation of quantum states. 




We consider the popular problem of multiple ion tomography, i.e.~the problem of estimating the joint quantum state of $m$ two-dimensional
systems (qubits), as encountered in ion trap quantum tomography, see~\cite{Guta, GLFBE10,butucea2015spectral, Blatt, acharya2015efficient} or~\cite{holevo2001statistical,nielsen} for textbooks on this problem. Such a system's \textit{quantum state} can be represented by a positive semi-definite unit trace complex matrix $\Theta$ (a \emph{density matrix}) of dimension $d := 2^m$.

For each individual qubit, the experimenter can measure one of the three Pauli observables described by the $2$ by $2$ \emph{Pauli matrices} $\sigma_1 , \sigma_2 , \sigma_3$, where
\begin{equation}
	\sigma^1= \left[
	\begin{array}{cc}
	0 & 1\\
	1 & 0
	\end{array}
	\right],\,
	\sigma^2=\left[
	\begin{array}{cc}
	0 & -i\\
	i & 0
	\end{array}
	\right],\,
	\sigma^3 =\left[
	\begin{array}{cc}
	1 &0\\
	0 & -1
	\end{array}
	\right],
\end{equation}	
and each of these measurements may yield one of two outcomes, denoted by
$+1$ and $-1$ respectively.

Therefore, a full experiment (i.e.~an experiment that describes the measurement for each of the $m$ qubits) is then defined by a setting $S = (s_1 , \ldots , s_m )$ where each $s_l \in \{\sigma_1, \sigma_2, \sigma_3\}$ for $l \leq m$, which specifies which of the $3$ Pauli observables is measured for each qubit. For each fixed setting $S$, the measurement produces random outcomes $O\in \{+1, -1\}^m$, with expected probability 
$$p_{O,S} = \mathrm{tr}(P_{O,S}\Theta),$$
where
$$P_{O,S} = \pi_{o_1,s_1} \otimes \ldots, \otimes \pi_{o_m,s_m},$$
where $\pi_{o_l,s_l}$ is the eigen projector of the the $2$ by $2$ Pauli matrix $s_l$ associated to the eigen value $o_l$ (we remind that the $2$ by $2$ Pauli matrices have eigen values of either $+1$ or $-1$).

Now set \begin{equation}
	\sigma_0=\left[
	\begin{array}{cc}
	1 &0\\
	0 & 1
	\end{array}
	\right],
\end{equation}
for the last $2\times 2$ Pauli matrix such that $(\sigma_i)_{i\in \{0,\ldots, 3\}}$ form an orthogonal basis of $\mathbb C^{2\times 2}$. Let $O$ be the outcome of an experiment given a setting $S = (s_1 , \ldots , s_m )$ (where each $s_l \in {\sigma_1, \sigma_2, \sigma_3}$). Write $\tilde S(E)$ for  a setting where a subset $E\subset \{1,\ldots, m\}$ of the $m$ matrices $s_l$ of this setting have been replaced by $\sigma_0$, and $\tilde O(E)$ for the outcome where the same subset $E$ of the $m$ elements $o_l$ have been replaced by $1$. Since the only eigen value of $\sigma_0$ is $1$, the outcome of the measurement of a qubit by $\sigma_0$ is always $1$. This implies in particular that the distribution of $\tilde O(E)$ as described above, is the same as the distribution of the outcome of an experiment when the measurement setting is $\tilde S(E)$. For this reason, measuring according to setting $S$ actually gives information about all settings $\tilde S(E)$ for any subset $E$ of $\{1, \ldots, m\}$. Thus instead of measuring all settings which are tensor products of $2\times2$ Pauli matrices $\sigma_0, \ldots, \sigma_3$, it is enough to measure all settings which are tensor products of $2\times2$ Pauli matrices $\{\sigma_1, \sigma_2, \sigma_3\}$, as they provide information about corresponding settings that involve $\sigma_0$. Therefore if one measures all $3^m$ settings that correspond to the settings $S = (s_1 , \ldots , s_m )$ where each $s_l \in \{\sigma_1, \sigma_2, \sigma_3\}$, we have observations about all measurements direction, and our measurement setting is complete.

Now we are interested in also dealing with situations where one does not want to observe all $3^m$ settings, and where one has only a number of settings $N \leq 3^m$.

We consider a random measurement setting as in~\cite{FGLE12}, i.e.~each measurement setting $S$ is chosen uniformly at random (each $s_l$ is chosen uniformly at random among $\sigma_1, \ldots, \sigma_3$). Let $N$ be the total number of measurement setting chosen in this random way. For each chosen measurement setting $S$, we perform $T$ repetition of the experiment and observe $T$ independent outcomes. So for each chosen measurement setting $S^i$ with $i \leq N$, we observe $T$ independent outcomes $O^{t,S^i}$ which are observations according to setting $S^i$.

\subsubsection{Expression of the outcomes in the trace regression model}

It is often convenient to express the 
information contained by a measurement $(S,O)$ in a way that involves
tensor products of $2\times2$ Pauli matrices, rather than their spectral
projections, see~\cite{FGLE12}. Indeed, the set of matrices that are created by $m$ tensor products of $2\times2$ Pauli matrices $\sigma_0, \ldots, \sigma_3$ is exactly the set of $2^m\times 2^m =d\times d$ Pauli matrices rescaled by $\sqrt{d}$ (introduced briefly in Remark~\ref{rem:gaussian}) i.e.~the $d\times d$ rescaled Pauli basis. Indeed let $f(O) = \prod_l o_l$, then one easily verifies that
\begin{align*} 
&\mathrm{tr} ((s_1 \otimes \dots \otimes s_{m} )\Theta)\\ 
&= \sum_{O\in\{1,-1\}^m} \Big(\prod_l o_l\Big) \mathrm{tr} \Big( (\pi_{s_1, o_1} \otimes \dots \otimes  \pi_{s_m, o_m})\Theta\Big)= \mathbb E_{O|S}\big(f(O)\big),
\end{align*}
where $\mathbb E_{O|S}$ is the expectation according to the outcome $O$ when measurement $S$ is chosen. In this sense, the measurement described by the $d\times d$ rescaled Pauli matrix  $P_S = s_{1}\otimes\dots\otimes s_m$ can be measured by the parity of the spins $f(O)$ that one gets when performing measurement $S$ : in fact $f(O)|S$ is a random variable such that its value is $1$ with probability $\big(\mathrm{tr} (P_S\Theta)+1\big)/2$ and $-1$ with probability $1-\big(\mathrm{tr} (P_S\Theta)+1\big)/2$ - and its expectation is $\mathrm{tr} (P_S\Theta)$ as noted. We write $\mathcal R(\mathrm{tr} (P_S\Theta))$ for this distribution.

\smallskip
Let us go back to our experimental setup. As remarked before, each of the $N$ quantum measurement settings according to setting $S^i$ (with $i \leq N$), where $S^i = (s_1^i , \ldots , s_m^i )$ and where each $s_l^i \in \{\sigma_1, \sigma_2, \sigma_3\}$ for $k \leq m$ provides us with $d$ information in the sense of our trace regression model, i.e.~we observe at each measurement $S^i$, for each replication $t \leq T$ and for all $E\subset \{1,\ldots m\}$
$$y_{S^i,E} =f(\tilde O^{t, S_i}(E)) \sim \mathcal R(\mathrm{tr} (P_{\tilde S^i(E)}\Theta)).$$

In our trace regression model, we can average the observations $f(\tilde O^{t, S_i}(E))$ and we have the the following averaged observations for any $i\leq s$ and for all $E\subset \{1,\ldots m\}$ 
$$\bar Y_{S^i,E} = \frac{1}{T}\sum_{t \leq T} \bar y_{S^i,E}^t =\mathrm{tr} (P_{\tilde S^i(E)}\Theta)) + \bar\epsilon_{S^i,E},$$
where $\bar y_{S^i,E}^t$ is the $t^{\text{th}}$ repetition (among $T$ iterations) of the observation and where $\bar \epsilon_{S^i,E}$ is the averaged noise and is such that $\mathbb E_{(O^{t,S^i})|S^i} \bar \epsilon_{S^i,E}= 0$, and such that $\bar \epsilon_{S^i,E}$ is sub-Gaussian has a sum of bounded random variables and satisfies $\mathbb E_{(O^{t,S^i}))|S^i} \exp(\lambda \bar \epsilon_{S^i,E}) \leq \exp(\lambda^2/T)$ for any $\lambda \geq 0$.

\smallskip
Now in order for our Assumption~\ref{ass:designbis} to be satisfied for rank $k$ matrices for a large enough number of settings $N$, we need to rescale our data. We set
$$Y_{S^i,E} = \sqrt{d} 3^{-|E|/2} \Big(\frac{3}{4}\Big)^{m/2} \bar Y_{S^i,E} = \sqrt{d} 3^{-|E|/2} \Big(\frac{3}{4}\Big)^{m/2}\mathrm{tr} (P_{\tilde S^i(E)}\Theta)) + \epsilon_{S^i,E},$$
where $|E|$ is the cardinality of $E$ and where $\epsilon_{S^i,E}$ is the rescaled noise and is such that $\mathbb E_{(O^{t,S^i})|S^i}  \epsilon_{S^i,E}= 0$, and such that $ \epsilon_{S^i,E}$ is sub-Gaussian and satisfies $\mathbb E_{(O^{t,S^i})|S^i} \exp(\lambda \epsilon_{S^i,E}) \leq \exp(\lambda^2 3^{-|E|} \Big(\frac{3}{2}\Big)^{m}/T)$ for any $\lambda \geq 0$. It is a direct consequence from the results of~\cite{liu} and from our Remark~\ref{rem:gaussian} that if $N \geq O(k^2d \log(d))$, then with high probability on the random draws of our settings we have that Assumption~\ref{ass:designbis} is satisfied for rank $k$ matrices. We can therefore apply our method and other low rank recovery methods such as trace regression methods to our rescaled data 
\begin{equation}\label{rescaleddata}
\Bigg(Y_{S^i,E}, \sqrt{d} 3^{-|E|/2} \Big(\frac{3}{4}\Big)^{m/2} P_{\tilde S^i(E)}\Bigg)_{i \leq N, E \subset \{1,\ldots, m\}}.
\end{equation}

\subsubsection{Experimental results}\label{subsub:experi}
We let $m \in \{4,5,6\}$ so that $d \in \{16, 32, 64\}$, and let $k \in \{1,2\}$ with $\alpha \in \{2,3,4,5\}$ and consider $N = \alpha k d$ measurement settings. These settings with small $k$ and $\alpha$ were chosen since we are more interested in the truncated measurement setting (such that $N \leq 3^m$). For the replication, we consider $T \in \{d, 10d\}$. Using the data (\ref{rescaleddata}), we estimate $\Theta$ by three methods---our proposed method (IHT), the truncated maximum likelihood estimator (MLE) for the high dimensional multiple ion tomography model as described in~\cite{acharya2015efficient}, and nuclear norm penalization (NNP) method computed using a gradient descent method \citep[e.g.][]{agarwal2012}.

We use the same tuning parameters as in (\ref{emprisk}), (\ref{upsilonr}), (\ref{T1}), and (\ref{Tr}) and also we select $\rho=1/2$ and the same stopping rule as we describe in Subsection \ref{s:sim}. For the stepsize of the gradient descent method used to compute the NNP estimator, we follow the recommendation in Subsection 3.1 in \cite{agarwal2012}.

\begin{center}
\begin{figure}
 \includegraphics[width=1\textwidth]{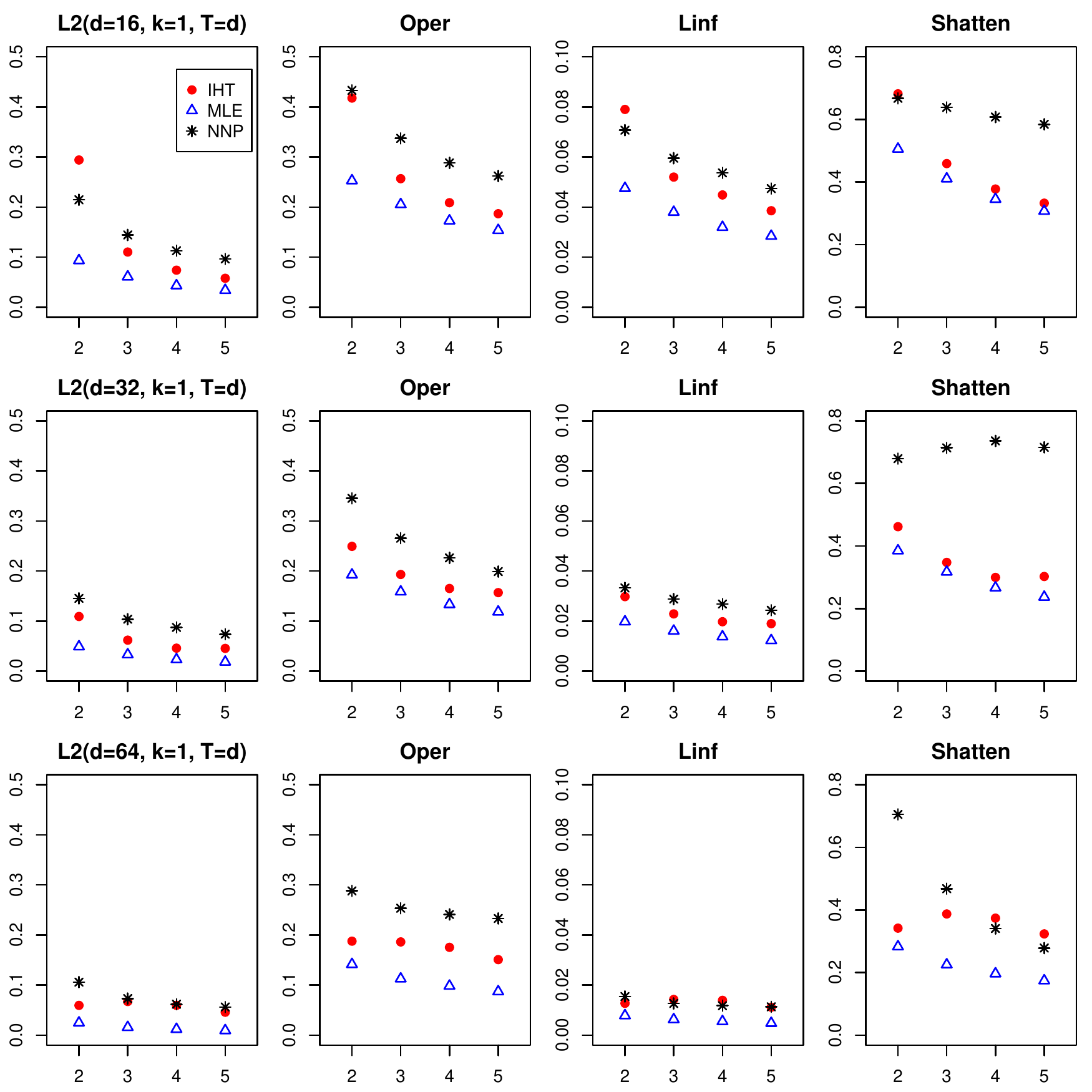}
\vspace{-0.2cm}
\caption{Squared Frobenius norm (L2), operator norm (Oper), entrywise $L_\infty$ norm (Linf) and Shatten $1$ norm (Shatten)  of $\hat \Theta - \Theta$ in function of $\alpha$ (and therefore in function og the number of settings $N$), for different values of $d$ for replication $T=d$ and $k=1$ using the three methods described.}  \label{fig:4}
\end{figure}
\end{center}

Figure \ref{fig:4} and \ref{fig:5} present the result when the number $T$ of replications is $d$ and when the true rank is 1 or 2 respectively, and for four different values of $d$. We provide average values of squared Frobenius norm, operator norm, entrywise $L_\infty$ norm, and Shatten $1$ norm of $\hat \Theta- \Theta$ averaged over 100 iterations. Red dots, blue blank triangles, and black asterisks are average value of IHT, MLE, and NNP, respectively. Intuitively when $\alpha$ increases these risks will decrease. Our estimator shows almost comparable result to the MLE except in the case where  $d=4$ and $\alpha=2$. In this case, IHT estimates $\Theta$ by $0$ a few times and pretty well for most cases so that on average the Frobenius norm is still large.  Figures \ref{fig:6} and \ref{fig:7} present the result when the replication is $10d$ and when the true rank is 1  or 2 respectively, and for four different values of $d$. They show similar patterns as for the cases $T=d$, but we can see that IHT performs well especially for $\alpha \in \{4,5\}$. An interesting feature that all these pictures illustrate is that IHT performs the best relatively to other methods for a large number of replication $T$, and for difficult problems with high $d$ and $k$ (see in particular the plots in Figure~\ref{fig:7} for large $d$). Note that our method is computationally more efficienct than the other two methods in the sense that when $d=64, \alpha = 5, k=2$, IHT takes about 40 seconds while MLE (and even NNP) takes about 2.5 minutes for one iteration, on a regular laptop, i.e.~it is about four times slower.

\begin{center}
\begin{figure}
 \includegraphics[width=1\textwidth]{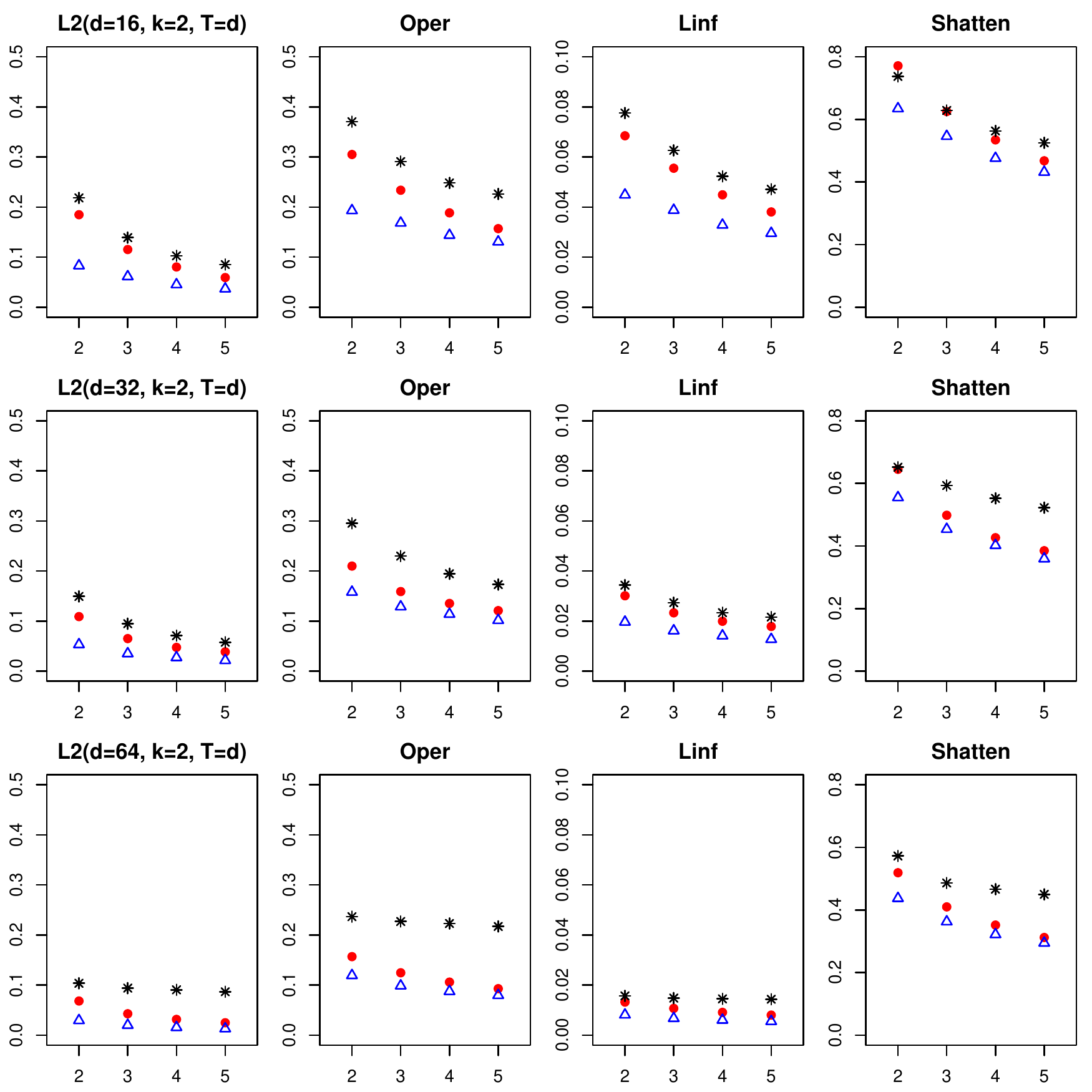}
\vspace{-0.2cm}
\caption{Squared Frobenius norm (L2), operator norm (Oper), entrywise $L_\infty$ norm (Linf) and Shatten $1$ norm (Shatten)  of $\hat \Theta - \Theta$ in function of $\alpha$, for different values of $d$ for replication $T=d$ and $k=2$ using three methods.}  \label{fig:5}
\end{figure}
\end{center}

\begin{center}
\begin{figure}
 \includegraphics[width=1\textwidth]{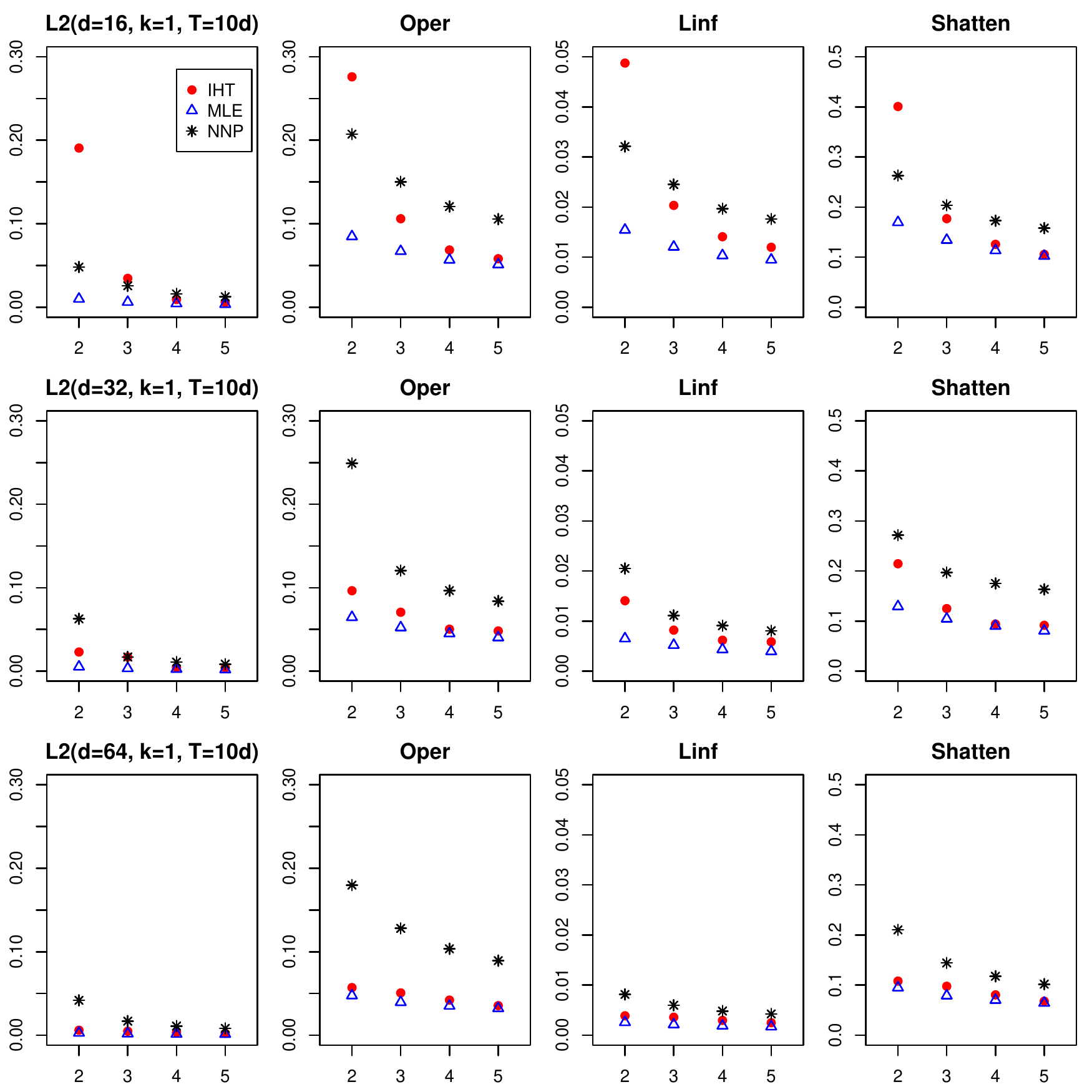}
\vspace{-0.2cm}
\caption{Squared Frobenius norm (L2), operator norm (Oper), entrywise $L_\infty$ norm (Linf) and Shatten $1$ norm (Shatten)  of $\hat \Theta - \Theta$ in function of $\alpha$, for different values of $d$ for replication $T=10d$ and $k=1$ using three methods.}  \label{fig:6}
\end{figure}
\end{center}

\begin{center}
\begin{figure}
 \includegraphics[width=1\textwidth]{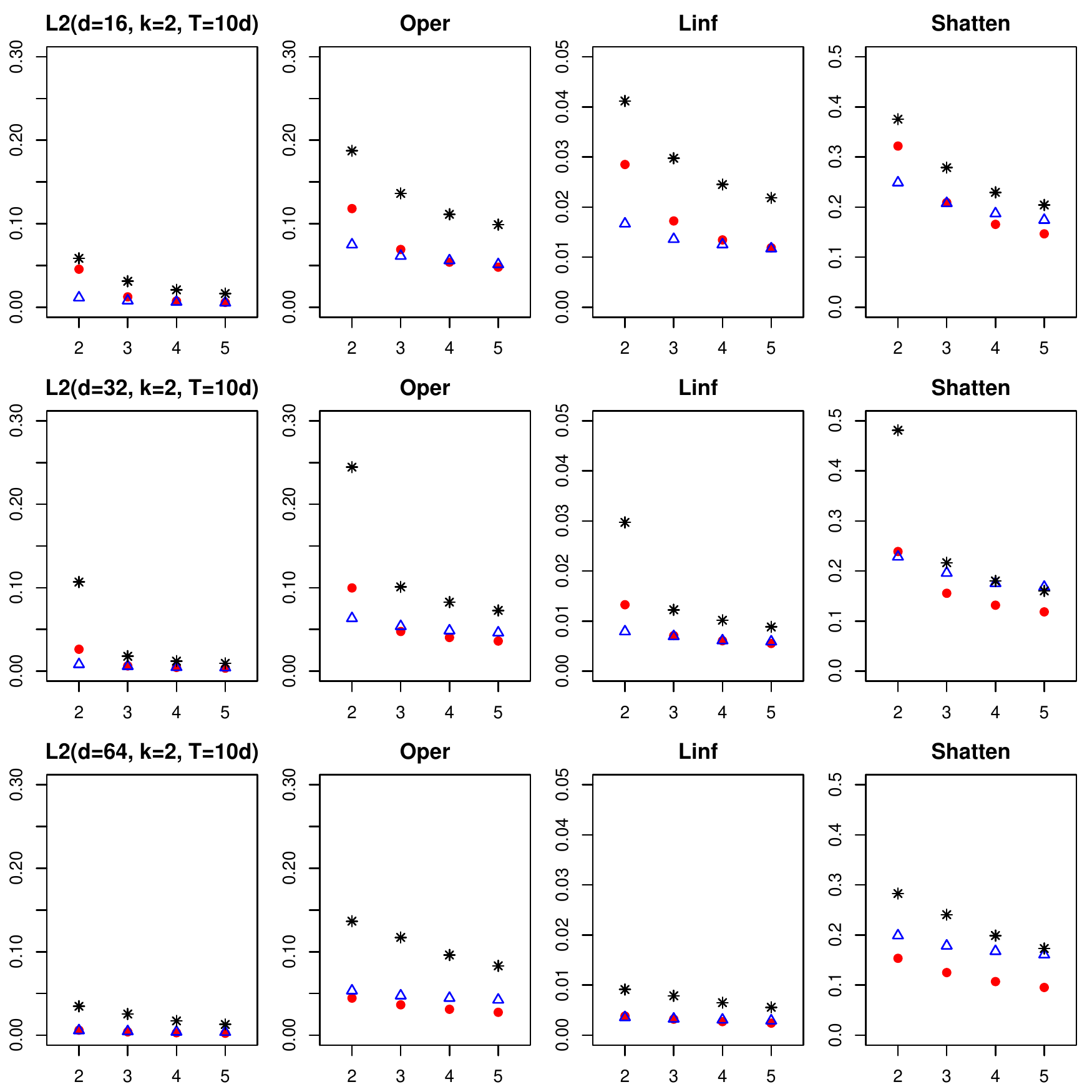}
\vspace{-0.2cm}
\caption{Squared Frobenius norm (L2), operator norm (Oper), entrywise $L_\infty$ norm (Linf) and Shatten $1$ norm (Shatten)  of $\hat \Theta - \Theta$ in function of $\alpha$, for different values of $d$ for replication $T=10d$ when $k=2$ using three methods.}  \label{fig:7}
\end{figure}
\end{center}

\section{Proofs}

\subsection{Preliminaries}
For convenience in writing the proofs, we introduce the following quantities.

We write, for integers $q,q'$, the vectorisation of a $q \times q'$ matrix (where $q'>0$) $A$ by stacking the rows of $A \in \rr^{q \times q'}$ as 
$$\text{vec}(A) = (A_{1,1}, A_{1,2},\ldots, A_{1,q'},A_{2,1},\ldots,A_{2,q'},\ldots,A_{q,1},\ldots,A_{q,q'})^T.$$
We write the Kronecker product between two matrices $A$ and $B$ as $A \otimes B$.

Consider the $n \times d^2$ matrix $\XX$ such that $\mathcal X_{i,M} = X^i_{m,m'}$ for $i \leq n$ and for $M = (m-1)d+ m'\leq d^2$ where $m,m'=1,\ldots,d$:
\begin{center}
$ \XX := \left[ \begin{array}{c}
 \text{vec}(X^1)^T\\
 \text{vec}(X^2)^T\\
\vdots \\
 \text{vec}(X^n)^T\\
\end{array}
\right]
= \left[ \begin{array}{cccccccc}
X^1_{1,1} & X^1_{1,2} & \ldots & X^1_{1,d} & \ldots & X^1_{d,1} & \ldots & X^1_{d,d} \\
X^2_{1,1} & X^2_{1,2} & \ldots & X^2_{1,d} & \ldots & X^2_{d,1} & \ldots & X^2_{d,d} \\
\vdots & \vdots & \vdots & \vdots & \vdots & \vdots & \vdots & \vdots \\
X^n_{1,1} & X^n_{1,2} & \ldots & X^n_{1,d} & \ldots & X^n_{d,1} & \ldots & X^n_{d,d}
\end{array}
\right].
$
\end{center}

 Let $\mathcal{R}(k)$ be the set of vectorization of matrices in $\MM(k)$, that is,
$\mathcal{R}(k) = \{\text{vec}(A): A \in \MM(k) \}$. If $A \in \MM(k)$, then $\RR(k)$ contains a vector $\bf a$ of dimension $d^2$ such that ${\bf{a}}_{M} = A_{m,m'}$ for $M = (m-1)d + m' \in \{1,\ldots,d^2\}$. 

Assumption~\ref{ass:designbis} can be rewritten as follows in this vectorized new notation.
\begin{ass}\label{ass:design}
Let $K \leq d$. For any $k\leq 2K$, it holds that
\begin{align*}
\sup_{A \in \mathcal{R}(k)}\Big|  \frac{1}{n}\|\mathcal X A\|_2^2 - \|A\|_2^2 \Big| \leq \tilde c_n(k)\|A\|_2^2,
\end{align*}
where $\tilde c_n(k)>0$.
\end{ass}

Assumption~\ref{ass:design} actually implies the following lemma that bounds the scalar products rather than the norms.
\begin{lem}\label{lem:tra}
If Assumption~\ref{ass:design} holds, then for any $k \leq K$, we have that
\begin{align}
&\sup_{A, B \in \mathcal{R}(k)^2} \Big|  \frac{1}{n}\langle \mathcal X A, \mathcal X B\rangle - \langle A, B\rangle \Big| \nonumber\\ 
&\leq 2\tilde c_n(2k)\|A\|_2 \|B\|_2 =: \|A\|_2 \|B\|_2 c_n(k). \label{RIPtrace}
\end{align}
\end{lem}
The proof of this lemma is in Subsection~\ref{proof:tra}.

\subsection{Proof of Theorem~\ref{th:mainthm2}}

Let $\Omega$ be the set of vectors of $\mathbb R^d$ and of norm $1$. 

\paragraph{1. Explicit writing of the quantities}

For any matrices $U,V$ of dimension $d \times m$ with $m\geq 1$, we set
\begin{align*}
 \tilde \gamma^r(U, V) &:= U^T \Big(\frac{1}{n}\sum_{i\leq n} (X^i)^T [Y_i  -  \mathrm{tr}\big( (X^i)^T\hat \Theta^{r-1}\big)] \Big)V \\
&= U^T \Big(\frac{1}{n}\sum_{i\leq n} (X^i)^T Y_i^r \Big)V =: U^T \hat \Psi^r V
\end{align*}
where $Y^{r}_i = Y_i  -  \mathrm{tr}\big( (X^i)^T\hat \Theta^{r-1}\big)$ and $\tilde \gamma^r(U,V) \in \rr^{m \times m}$. Also we set $\Psi^r:= \Theta - \hat \Theta^{r-1}$ and
$$\gamma^r(U, V) := U^T \Big(\Theta - \hat \Theta^{r-1}\Big) V = U^T \Psi^r V \in \rr^{m \times m}.$$
Note that $Y_i^r = \mathrm{tr}\big((X^i)^T \Psi^r \big)+\epsilon_i$ by linearity of the trace.

Let $\mathbf{u},\mathbf{v} \in \Omega$, then $\tilde \gamma^r(\mathbf{u},\mathbf{v})$ is a scalar:
\begin{align*}
\tilde \gamma^r(\mathbf{u}, \mathbf{v}) &= \sum_{m, m'\leq d} \mathbf{u}_{m} \mathbf{v}_{m'}\frac{1}{n}\sum_{i\leq n} X^i_{m',m}Y_i^r\\
&= \sum_{m, m'\leq d} \mathbf{u}_{m} \mathbf{v}_{m'} \frac{1}{n}\sum_{i \leq n} X^i_{m',m}\Big(\mathrm{tr}\big((X^i)^T\Psi^r\big) + \epsilon_i\Big)\\
&= \sum_{m, m'\leq d} \mathbf{u}_{m} \mathbf{v}_{m'}\frac{1}{n}\sum_{i\leq n} X^i_{m',m}\Big(\sum_{k,k'\leq d} X^i_{k,k'}\Psi_{k,k'}^r + \epsilon_i\Big).
\end{align*}

Let $\mathcal U$ be the column vector of dimension $d^2$ such that $\UU = \text{vec}(\mathbf{u} \mathbf{v}^T) = \mathbf{u} \otimes \mathbf{v}$, that is, 
$\UU_{M} = \mathbf{u}_{m} \mathbf{v}_{m'}$ for $M = (m-1)d+ m'$. Note that $\mathcal U \in \RR(1)$, and that  $\|\mathcal U\|_2^2 = \sum_{m,m' \leq d} (\mathbf{u}_{m} \mathbf{v}_{m'})^2  =  \sum_{m} (\mathbf{u}_{m})^2 \sum_{m'}(\mathbf{v}_{m'})^2 = 1$. Consider the $n \times d^2$ matrix $\XX$ such that $\mathcal X_{i,M} = X^i_{m,m'}$ for $i \leq n$ . Consider the column vector $\psi^r \in \rr^{d^2}$ where $\psi^r = \text{vec}(\Psi^r)$. Then we have
\begin{align}
\tilde \gamma^r(\mathbf{u}, \mathbf{v}) &=\frac{1}{n} \sum_{M \in \{1, \ldots, d^2\}} \mathcal U_{M}\sum_{i\leq n} (\mathcal X_{i,M})^T \Big(\sum_{K \in \{1, \ldots, d^2\}} \mathcal X_{i,K} \mathcal \psi_K^r + \epsilon_i \Big)\nonumber\\
&=\frac{1}{n} (\mathcal U)^T (\mathcal X)^T\Big( \mathcal X \psi^r + \epsilon\Big) \nonumber\\
&= \frac{1}{n}\Big( \langle \mathcal X \mathcal U, \mathcal X \psi^r\rangle + \langle \mathcal X \mathcal U, \epsilon \rangle \Big),\label{simple1}
\end{align}
where here $\langle .,.\rangle$ is the classic vectorial scalar product on $\mathbb R^n$. 
Also by definition of $\UU$ and $\psi^r$, we have
\begin{equation}\label{simple2}
\gamma^r(\mathbf{u},\mathbf{v}) = \langle \UU, \psi^r \rangle.
\end{equation}

The last equation implies that
\begin{align}
&\sup_{\mathbf{u}, \mathbf{v} \in \Omega} |\tilde \gamma^r(\mathbf{u}, \mathbf{v}) - \gamma^r(\mathbf{u}, \mathbf{v})| \nonumber \\
&\leq \sup_{\mathcal U \in \RR(1), \|\mathcal U\|_2=1} \Big|  \frac{1}{n}\langle \mathcal X \mathcal U, \mathcal X \psi^r\rangle - \langle \UU, \psi^r \rangle\Big|+ \sup_{\mathcal U \in \RR(1), \|\mathcal U\|_2=1}  \Big| \frac{1}{n} \langle \mathcal X \mathcal U, \epsilon \rangle\Big|. \label{eq:dec}
\end{align}

\paragraph{2. Bound on the stochastic term}

We first bound the second term in (\ref{eq:dec}) with the following lemma.
\begin{lem}\label{helo}
Assume that $c_n(1) \leq 1$ (note that $c_n(1) \leq c_n(K)$ for $K \geq 1$). It holds with probability larger than $1-\delta$ that
\begin{equation}\label{eq:1}\sup_{A \in \RR(1)} \Big| \frac{1}{n}  \langle \mathcal X A, \epsilon \rangle \Big| \leq   C \|A\|_2 \sqrt{d \frac{\log(1/\delta)}{n}} =: \|A\|_2 \upsilon_n, 
\end{equation}
where $C$ is an universal constant.
\end{lem}
Its proof is in Subsection~\ref{proof:helo}. 
Lemma \ref{helo} implies that on an event of probability larger than $1-\delta$, we can bound the stochastic term in (\ref{eq:dec}) 
\begin{equation}\label{eq:1_1}
\sup_{\mathcal U \in \RR(1), \|\mathcal U\|_2=1}  \Big| \frac{1}{n} \langle \mathcal X \mathcal U, \epsilon \rangle\Big| \leq \upsilon_n.
\end{equation}
Let $\xi$ be an event of probability larger than $1-\delta$ where the above holds.

\paragraph{3. Bound on the first term in (\ref{eq:dec}) provided that the rank $k^r$ of $\Psi^r$ is smaller than $2k$}

Let us assume, only for this Paragraph 3.~of the proof, that the rank $k^r$ of $\Psi^r$ is smaller than $2k \leq K$. By Lemma~\ref{lem:tra}, we can apply Equation~\eqref{RIPtrace} (since $k^r \leq 2k \leq K$), and combining this with the fact that $\|\UU \|_2 =1$, we have
\begin{equation}\label{eq:2}\sup_{\mathcal U \in \RR(1), \|\UU \|_2 = 1} \Big| \frac{1}{n}\langle \mathcal X \mathcal U, \mathcal X \psi^r\rangle - \langle \mathcal U, \psi^r\rangle\Big| \leq c_n(2k) \|\psi^r\|_2.\end{equation}


By combining Equation~\eqref{eq:dec}, \eqref{eq:1_1} and~\eqref{eq:2}, and using $\| \psi^r\|_2 = \|\Psi^r\|_2 $, we then have in the case that $k^r \leq 2k\leq K$ that on $\xi$ 
\begin{align}\label{eq:linf}
\sup_{\mathbf{u},\mathbf{v} \in \Omega} | \tilde \gamma^r(\mathbf{u},\mathbf{v}) - \gamma^r(\mathbf{u},\mathbf{v})| \leq  c_n(2k)\|\Psi^r\|_2  + \upsilon_n.
\end{align}
Since the previous result holds in the worst case of $ \mathbf{u},\mathbf{v} \in \Omega$, we directly have on $\xi$ the corresponding entrywise result whenever $k^r \leq 2k$
\begin{align}\label{eq:linf2}
\sup_{U, V \in \MM_\Omega} \| \tilde \gamma^r(U,V) - \gamma^r(U,V)\|_{\infty} \leq  c_n(2k) \|\Psi^r\|_2  + \upsilon_n.
\end{align}
By definition, we know $\tilde \gamma^r(U,V) = U^T \hat \Psi^r V$ and $\gamma^r(U,V)= U^T \Psi^r V = U^T(\Theta - \hat \Theta^{r-1})V$, which gives on $\xi$ whenever $k^r \leq 2k$,
\begin{align}\label{eq:linf4}
\sup_{U, V \in \MM_\Omega} \| U^T (\hat \Psi^r + \hat \Theta^{r-1} - \Theta) V\|_{\infty} \leq  c_n(2k)\|\Psi^r\|_2  + \upsilon_n.
\end{align}

Note also by definition of the thresholding process, the matrix
$$D = (U^{r})^T  (\hat \Psi^r + \hat \Theta^{r-1}) V^r - \lfloor (U^r)^T (\hat \Psi^r + \hat \Theta^{r-1}) V^r \rfloor_{T_r},$$
is such that it is diagonal with all diagonal elements smaller than $T_r$.
Let $\tilde U = (U^r)^TU \in \MM_\Omega$ and $\tilde V = (V^r)^T V\in \MM_\Omega$ for $U,V \in \MM_\Omega$. By elementary calculations, we have
$$\tilde U^T D \tilde V = \big(\sum_k \tilde U_{k,i} D_{k,k} \tilde V_{k,j}\big)_{i,j},$$
and so we have that
\begin{align*}
 \|\tilde U^T D \tilde V\|_{\infty} &\leq  \sup_{i,j} |\sum_k \tilde U_{k,i} D_{k,k} \tilde V_{k,j}|\\ 
&\leq T_r\sup_{i,j} \sum_k |\tilde U_{k,i} \tilde V_{k,j}| \leq T_r\sup_{i,j} \sqrt{\|\tilde U_{,i}\|_2 \|\tilde V_{,j}\|_2} = T_r.
\end{align*}

By definition of $\hat \Theta^r$, we have
\begin{align*}\tilde U^T D \tilde V &= U^T  (\hat \Psi^r + \hat \Theta^{r-1}) V -U^T U^{r}\lfloor (U^{r})^T (\hat \Psi^r + \hat \Theta^{r-1}) V^{r}\rfloor_{T_r} (V^{r})^T V\\
 &= U^T  (\hat \Psi^r + \hat \Theta^{r-1}) V - U^T \hat \Theta^{r} V,
\end{align*}
so this implies that
$$\sup_{U,V \in \MM_\Omega}  \|U^T  (\hat \Psi^r + \hat \Theta^{r-1}) V -U^T \hat \Theta^r V \|_{\infty} \leq T_r.$$
Combining this with Equation~\eqref{eq:linf4}, we obtain that on $\xi$ and whenever $k^r \leq 2k$
\begin{align}
\sup_{U,V \in \MM_\Omega} \| U^T \Psi^{r+1} V\|_{\infty} &= \sup_{ U, V \in \MM_\Omega} \| U^T (\Theta-\hat \Theta^{r} ) V\|_{\infty} \nonumber \\ 
 &\leq  c_n(2k)\|\Psi^r\|_2  + \upsilon_n + T_r. \label{eq:linf5}
\end{align}

\paragraph{4. Induction}


We now stop assuming that the rank $k^r$ of $\Psi^r$ is smaller than $2k$, and we consider the general case.

We are going to prove by induction that on $\xi$, for any integer $r\geq 1$, we have that (i) the rank of $\hat \Theta^{r-1}$ is smaller than $k$, and (ii) $\|\Psi^r\|_2 \leq 2\sqrt{2k }T_{r-1} := C_r$.

For $r = 1$, since $\hat \Theta^0 = 0$, then its rank is $0$ and is therefore bounded by $k$  and (i) is satisfied. Moreover, since $T_0 = B \geq \|\Theta\|_2 = \|\Theta - \hat \Theta^0\|_{2} = \|\Psi^1\|_{2}$, then (ii) is satisfied as well.

Now assume that (i) and (ii)  hold on $\xi$ for a given $r$ (as it holds for $r=1$ not only on $\xi$ but on the entire probability space). By induction assumption (i), we have that on $\xi$ the rank of $\hat \Theta^{r-1}$ is smaller than $k$, which implies that the rank of $\Psi^{r} = \Theta - \hat \Theta^{r-1}$ is smaller than $k + k = 2k$. 

Because we have that $k^r \leq 2k$, Equation~\eqref{eq:linf5} applies and on $\xi$ 
\begin{align}
\sup_{U, V \in \MM_\Omega} \| U^T  \Psi^{r+1} V\|_{\infty} &\leq  c_n(2k) C_r + \upsilon_n + T_r \nonumber\\
&\leq  2\sqrt{2k } c_n(2k) T_{r-1} + \upsilon_n + T_r \leq 2T_r,\label{eq:linf7}
\end{align}
by definition of $T_r$ and since $2\sqrt{2k}c_n(2k) \leq 4 \sqrt{K}\tilde c_n(2K)$ (since $2k \leq K$). 
Moreover, in the same way, we have that on $\xi$ (see Equation~\eqref{eq:linf4} since $k^r \leq 2k$)
\begin{align}\label{eq:linf8}
\sup_{U, V \in \MM_\Omega} \| U^T (\hat \Psi^r + \hat \Theta^{r-1} - \Theta) V\|_{\infty} \leq  c_n(2k)C_r + \upsilon_n \leq T_r.
\end{align}

Let us now state the following lemma.
\begin{lem}\label{lem:mati}
Let $M \in \rr^{d_1 \times d_2}$ be a matrix (with $d_1\geq d_2$), with singular values $(\lambda_j)_j$ ordered in decreasing order (all positive). For any $j \leq d_2$ and for any collection of orthogonal vectors $(\mathbf{w}^{j'})_{j'\leq j-1}$, we have 
$$\lambda_{j} \leq  \sup_{\mathbf{u}\in \rr^{d_1}, \mathbf{v}\in \rr^{d_2} : \| \mathbf{u}_2\|=1, \|\mathbf{v}\|_2=1,\mathbf{u} \perp (\mathbf{w}^{j'})_{j'\leq j-1}} |\mathbf{u}^T M \mathbf{v}|.$$
\end{lem}



Write $(\hat \lambda_j^r)_j$ for the singular values of $\hat \Psi^r + \hat \Theta^{r-1}$ ordered in decreasing order and all positive (and $U^{r},V^{r}$ for the diagonalising matrices). Let $U^{*},V^{*}$ be the matrices that diagonalise $\Theta$ and order its singular values in decreasing order on the diagonal and write $(\lambda_j^*)_j$ for its singular values (all positive). By Lemma~\ref{lem:mati}, we know that, for any $j \leq d$,
$$  \hat \lambda_{j}^r \leq  \sup_{\mathbf{u}, \mathbf{v} \in \Omega : \mathbf{u} \perp (U^{*}_{l,.})_{l\leq j-1}} |\mathbf{u}^T (\hat \Psi^r + \hat \Theta^{r-1}) \mathbf{v}|.$$

Therefore, on $\xi$, by Equation~\eqref{eq:linf8}, we know that for any $j \leq d$
$$\hat \lambda_{j}^r \leq  \sup_{\mathbf{u}, \mathbf{v}\in \Omega : \mathbf{u} \perp (U^{*}_{l,.})_{l\leq j-1}} |\mathbf{u}^T  \Theta \mathbf{v}| + T_r  = \lambda_{j}^* + T_r.$$
So since all $\hat \lambda_{j}^r$ that are smaller than $T_r$ are thresholded for constructing $\hat \Theta^r$ (we remind that the $\hat \lambda_j^r$ are the diagonal elements of the diagonal matrix $(U^r)^T  (\hat \Psi^r + \hat \Theta^{r-1}) V^r$, that is thresholded at level $T_r$ in the construction of $\hat \Theta^r$), it means that on $\xi$, the rank of $\hat \Theta^r$ is smaller than the rank of $\Theta$, i.e.~it is smaller than $k$. This proves the first part of the induction (i).

Now let $\breve U^r, \breve V^r$ be the matrices that diagonalise $\Psi^{r+1}$, and let $D^{r+1} = (\breve U^r)^T \Psi^{r+1} \breve V^r$. By \eqref{eq:linf7}, we have that on $\xi$ 
\begin{align*}
\|D^{r+1}\|_{\infty} = \| (\breve U^r)^T  \Psi^{r+1} \breve V^r\|_{\infty} \leq 2T_r.
\end{align*}
Now since the rank of both $\hat \Theta^r$ and $\Theta$ are smaller than $k$ on $\xi$, we know that the rank of $\Psi^{r+1}$, and thus of $D^{r+1}$, is smaller than $2k$. Therefore, we have since $D^{r+1}$ is diagonal and has therefore only $2k$ non-zeros elements that on $\xi$
\begin{align*}
\|D^{r+1}\|_{2} \leq 2\sqrt{2k}T_r,
\end{align*}
which implies that on $\xi$, since the Frobenius norm is invariant by rotation
$$\|\Psi^{r+1}\|_2 = \|D^{r+1}\|_{2}  \leq 2\sqrt{2k}T_r.$$
This concludes part (ii) of the induction and therefore, it concludes the induction.

\paragraph{5. Conclusion}

By the previous induction, we know that on $\xi$, we have
\begin{align}\label{eq:cocococo}
\sup_{\Theta \in \RR(k), \|\Theta\|_2 \leq B, U,V \in \MM_\Omega^2} \| U^T  \Psi^{r+1} V\|_{\infty} \leq  2T_r,
\end{align}
and also that
$$\mathrm{rank}(\hat \Theta^r) \leq k,$$
and also that for any $p>0$
$$\sup_{\Theta \in \RR(k), \|\Theta\|_2 \leq B} \|\Psi^{r+1}\|_{S_p} \leq 2(2k)^{1/p} T_r.$$

This concludes the proof since for $r$ larger than $c_l\log(n)$ with $c_l$ a large enough constant, we have by definition of the sequence $T_r$ that
$$T_r \leq 2\upsilon_n \leq 2C\sqrt{\frac{d\log(1/\delta)}{n}}.$$

\subsection{Proof of Lemma~\ref{lem:tra}}\label{proof:tra}
First, note that for $A \in \mathcal{R}(k), B \in \mathcal{R}(k)$, we have
$$
\|A\|_2\|B\|_2\left| \Big\langle \XX \frac{A}{\|A\|_2}, \XX \frac{B}{\|B\|_2} \Big\rangle - \Big\langle \frac{A}{\|A\|_2},\frac{B}{\|B\|_2} \Big\rangle \right| = |\langle \XX A, \XX B \rangle - \langle A,B \rangle |.
$$
Thus, without loss of generality, we consider $\|A\|_2 = \|B\|_2 = 1$. We know that
$$ \langle \mathcal X A, \mathcal X B\rangle= \frac{\|\mathcal XA\|_2^2 + \|\mathcal X B\|_2^2 - \|\mathcal X (A-B)\|_2^2}{2},$$
and
$$\langle A, B\rangle =  \frac{\|A\|_2^2 + \|B\|_2^2 -  \|A-B\|_2^2}{2}.$$
This gives
\begin{align*}
 \Big|&\frac{1}{n}\langle \mathcal XA, \mathcal X B\rangle - \langle A, B\rangle\Big|\\
&\leq \frac{1}{2}\Big(\Big|\frac{1}{n}\|\mathcal X A\|_2^2 - \|A\|_2^2\Big| + \Big|\frac{1}{n}\|\mathcal X B\|_2^2- \|B\|_2^2 \Big| +  \Big|\frac{1}{n}\|\mathcal X(A-B)\|_2^2 - \|A-B\|_2^2\Big|\Big).
\end{align*}
By Assumption~\ref{ass:design}, using $A-B \in \RR(2k)$, we have for $k \leq K$,
$$\big|\langle \mathcal X A, \mathcal X B\rangle - \langle A, B\rangle\big|\leq \frac{1}{2}\left(\tilde c_n(k)+ \tilde c_n(k)+2 \tilde c_n(2k) \right)\leq 2 \tilde c_n(2k) =: c_n(k).$$
This concludes the proof.

\subsection{Proof of Lemma~\ref{helo}}\label{proof:helo}

Since $\epsilon \sim \mathcal N(0, I_n)$, we have that
$$\frac{1}{n}  \langle \mathcal X A, \epsilon \rangle \sim \mathcal N(0, \frac{1}{n^2} \|\mathcal X A\|_2^2).$$
This implies that (using a Gaussian tail probability $P(|X|>x) \leq e^{-x^2/2}$ for $x>0$ when $X \sim \mathcal{N}(0,1)$) with probability larger than $1-\delta$
\begin{align}\label{eq:ber21}
\Big| \frac{1}{n}  \langle \mathcal X A, \epsilon \rangle \Big| \leq \frac{1}{n} \|\XX A\|_2 \sqrt{\frac{1}{2}\log(1/\delta)}.
\end{align}

Since $\XX$ satisfies the Assumption~\ref{ass:design} with $K \geq 1$, we have that
\begin{align*}
\sup_{A \in \RR(2)}\Big|\frac{1}{n} \| \mathcal X A\|_2^2 - \|A\|_2^2 \Big| \leq  \tilde c_n(2) \| A\|_2^2,
\end{align*}
which implies that for any $A \in \RR(2)$, we have
\begin{align}\label{cicir}
\|\mathcal X A\|_2 \leq \sqrt{n}\| A\|_2 \sqrt{1+ \tilde c_n(2)}.
\end{align}

Equation~\eqref{cicir} implies together with Equation~\eqref{eq:ber21} that for a matrix $A \in \RR(2)$, with probability larger than $1-\delta$,
\begin{align}\label{eq:ber2}
\Big| \frac{1}{n}  \langle \mathcal X A, \epsilon \rangle \Big| \leq \sqrt{\frac{1+ \tilde c_n(2)}{2}} \|A\|_2 \sqrt{\frac{\log(1/\delta)}{n}} =: \|A\|_2 v_n(\delta),
\end{align}
where $v_n(\delta) = \sqrt{\frac{1+\tilde c_n(2)}{2}}\sqrt{\frac{\log(1/\delta)}{n}}
$.


To obtain the bound for the supremum of the quantity in (\ref{eq:ber2}) over all $A \in \{A, A \in \RR(1), \|A\|_{2} \leq 1\} =: \mathcal{A}(1)$, we consider the approximating set $\BB_0 \subseteq \BB_1 \subseteq \ldots $ whose property is described as follows. Let $\mathcal B_0 = \{0\}$. Let, for any $i \in \mathbb N^*$, $\mathcal B_i$ be a $2^{-i}$ covering set of $\mathcal{A}(1)$.
Here we use a classic result \citep[Lemma 3.1]{candes2011}, saying that the $\upsilon$-covering numbers of $\mathcal{A}(1)$ is bounded by $(C/\upsilon)^{2d+1}$.

Thus the cardinality of $\mathcal B_i$ is smaller than $(C2^i)^{2d+1}$. Let $\tilde \xi$ be the event such that for all $i,j \in \mathbb N^2$ and for each vector in $\mathbf{u},\mathbf{v} \in \BB_i\times \BB_j$, it holds that
\begin{align}\label{eq:xi}
\Big| \frac{1}{n} \langle \XX (\mathbf{u}-\mathbf{v}), \epsilon\rangle \Big|
\leq \|\mathbf{u}-\mathbf{v}\|_2 v_n(\delta_{i,j}),
\end{align}
where $\delta_{i,j} = \delta (C'2^{\max(i,j)})^{-7d}$, where $C'>2C$ is a large constant.
By Equation~\eqref{eq:ber2}, and since $\mathbf{u}-\mathbf{v} \in \mathcal{R}(2)$ we know that $(\ref{eq:xi})$ holds with probability $1-\delta_{i,j}$ for each $i,j$ and for each vector $\mathbf{u},\mathbf{v} \in \BB_i\times \BB_j$.
By a union bound, we have that
\begin{align*}
\mathbb P(\tilde \xi) &\geq 1 - \sum_{i, j  \in \mathbb N^2} |\mathcal B_i||\mathcal B_j| \delta_{i,j} \\
&\geq 1 - 2\delta \Big(\sum_i  (C2^i)^{2d+1}\sum_{j \leq i} (C2^j)^{2d+1}(C'2^{\max(i,j)})^{-7d}\Big) \\
&\geq 1 - 2C^{4d+2} (C')^{-7d} \delta \Big(\sum_i  2^{4di+2i} i (2^i)^{-7d}\Big)\\
&\geq 1 - 2C^{4+2} (C')^{-7d} \delta  \Big(\sum_i  i 2^{-i} \Big) = 1- 2C^{4d+2}(C')^{-7d} \delta \\
&\geq 1-\delta,
\end{align*}
since $C'>2C$.

Let now $A \in \RR(A)$ such that $\|A\|_2 = 1$. It is possible to write $A$ as
$$ A =  \sum_{i=1}^{\infty} (\mathbf{u}_i - \mathbf{u}_{i-1}),$$
where each $\mathbf{u}_i$ belongs to $\mathcal B_i$, and where the $(\mathbf{u}_i)_i$ are such that $\|\mathbf{u}_i - \mathbf{u}_{i-1}\|_2 \leq 2^{-i}$. We have on $\tilde \xi$ that
\begin{align*}
\sup_{A \in \RR(1)} \Big|  \frac{1}{n}  \langle \mathcal X A, \epsilon \rangle \Big| &= \Big| \frac{1}{n}  \big\langle \mathcal X \big(\sum_{i=1}^{\infty} (\mathbf{u}_i - \mathbf{u}_{i-1}) \big), \epsilon \big\rangle  \Big|
= \Big| \frac{1}{n} \sum_{i=1}^{\infty}  \langle \mathcal X (\mathbf{u}_i - \mathbf{u}_{i-1}), \epsilon \rangle \Big|\\
&\leq \sum_{i=1}^{\infty}  \Big| \frac{1}{n}  \langle \mathcal X (\mathbf{u}_i - \mathbf{u}_{i-1}), \epsilon \rangle \Big|\\
&\leq \sum_{i=1}^{\infty}  \|\mathbf{u}_i - \mathbf{u}_{i-1}\|_2   v_n(\delta_{i,i-1})\\
&\leq \sum_{i=1}^{\infty}  2^{-i } v_n(\delta_{i,i-1})\\
&\leq \sum_{i=1}^{\infty}   2^{-i } C \sqrt{\frac{\log((C'2^{i-1})^{7d}/\delta)}{n}} 
\leq \tilde C \sqrt{d \frac{\log(1/\delta)}{n}}.
\end{align*}
This concludes the proof.

\subsection{Proof of Lemma \ref{lem:mati}}
Let $(\mathbf{u}^k)_k \in \rr^{d_1},(\mathbf{v}^k)_k \in \rr^{d_2}$ be the singular vectors of $M$, i.e.~$M \mathbf{v}^k = \lambda_k \mathbf{u}^k$. Let $E = \mathrm{span}((\mathbf{v}^k)_{k\leq j})$. The dimension of $E$ is $j$. Let now $F$ be the vectorial sub-space that is orthogonal to $\mathrm{span}((\mathbf{w}^{j'})_{j'\leq j-1})$. Its dimension is $d_2 - j+1$. Since $\mathrm{dim}(E) + \mathrm{dim}(F) = d_2+1$, there is at lest one unitary vector in $E \bigcap F$. Let $\mathbf{h} \in \rr^{d_2}$ be this vector, since it is in $E$, it can be written as
$$\mathbf{h} = \sum_{k\leq j} h_k \mathbf{v}^k$$
where for $k=1,\ldots, j$, we have $h_k \geq 0$ and $\sum_{k}h_k^2 = 1$.
Therefore, we have that
$$M\mathbf{h} = \sum_{k\leq j} \lambda_k h_k \mathbf{u}^k.$$
Consider $\mathbf{g} = M\mathbf{h}/\|M\mathbf{h}\|_2$. So we have that
\begin{align*}
 |\mathbf{g}^TM\mathbf{h}| &=\frac{\left\{(M\mathbf{h})^T(M\mathbf{h})\right\}}{ \|M\mathbf{h}\|_2} =  \|M\mathbf{h}\|_2 \\
&=  \sqrt{\sum_{k\leq j} \lambda_k^2 h_k^2} \geq \sqrt{\min_{k\leq j}(\lambda_k^2) \times \sum_k h_k^2} = \lambda_j,
\end{align*}
since the $(\mathbf{u}^k)_k$ are orthonormal, and since the singular values are positive and ordered in decreasing order. This concludes the proof.

\subsection{Proof of Theorem~\ref{prop:srule}}

From the proof of Theorem~\ref{th:mainthm2}, Equation~\eqref{eq:cocococo}, we know that the estimator $\hat \Theta^{\hat r}$ satisfies
$$\|\hat \Theta^{\hat r} - \Theta\|_{S} \leq T_{\hat r}\leq  1.1 \frac{1}{1-\rho} v_n = O(\sqrt{d/n}),$$
i.e.~the desired result in operator norm and from which the results on the rank and in the other Schatten norm follow because of the thresholding. Also since the sequence $T_r$ is an arithmetico-geometrical sequence converging to $\frac{1}{1-\rho} v_n$, of arithmetic term $v_n$ and of geometric term $\rho$, it is clear that $\hat r$ is such that
$$\rho^{\hat r-1} T_0 \geq \frac{0.1}{1-\rho} v_n,$$
i.e.
$$\hat r \leq 1 + \frac{\log\Big(10(1-\rho)T_0/(v_n)\Big)}{\log(1/\rho)} \leq O(\log(n)).$$
This concludes the proof.



\subsection{Proof of Theorem~\ref{thm:asymnorm2}}

By definition, we have that
\begin{align*}
\sqrt{n} (\hat \Theta - \Theta ) &= \sqrt{n}(\hat \Theta^r - \Theta) + \frac{1}{\sqrt{n}} \sum_{i=1}^n (X^i)^T\big(\text{tr}((X^i)^T(\Theta - \hat \Theta^r )) + \epsilon_i \big)\\
&= \sqrt{n}(\hat \Theta^r - \Theta) + \frac{1}{\sqrt{n}} \sum_{i=1}^n (X^i)^T\big(\text{tr}((X^i)^T(\Theta - \hat \Theta^r )) \big) \\
& \  \ \ \ \ + \frac{1}{\sqrt{n}} \sum_{i=1}^n (X^i)^T \epsilon_i \\
&= \Delta + Z.
\end{align*}

Let $m, m' \leq d$ and let $\mathbf{u}^m$ be the vector with all element equal to $0$ except the $m^\text{th}$ entry which is equal to $1$, and we consider that $\mathcal U^{m,m'} = \text{vec}(\mathbf{u}^m (\mathbf{u}^{m'})^T)$. We have by definition and using representations (\ref{simple1}) and (\ref{simple2}) that
\begin{align*}
\Delta_{m,m'} 
&=\sqrt{n}  \Big( \frac{1}{n}\langle \mathcal X \mathcal U^{m,m'}, \mathcal X \psi^{r+1} \rangle - \langle \mathcal{U}^{m,m'}, \psi^{r+1} \rangle \Big)
\end{align*}
and
$$
Z_{m,m'} =  \frac{1}{\sqrt{n}} \sum_{i=1}^n (X^i)_{m,m'}\epsilon_i.
$$
Note that given $X^{i}, \ i=1,\ldots,n$, 
$$
\text{Var}(Z_{m,m'}) = \frac{1}{n}\sum_{i=1}^n (X^i)_{m,m'}^2
$$
and
$$
\text{Cov}(Z_{j,j'},Z_{l,l'}) = \frac{1}{n}\sum_{i=1}^n (X^i)_{j,j'} (X^i)_{l,l'}.
$$

The following Lemma is a concentration inequality that holds in Gaussian design.


\begin{lem}\label{lem:concer}
Assume that the design is Gaussian. Let $A \in \MM$. We have that with probability larger than $1-\delta$ (on the design)
\begin{align*}
\Big|\frac{1}{n}\|\mathcal X A\|_2^2 - \|A\|_2^2 \Big|\leq C\|A\|_2^2\big(\sqrt{\frac{\log(1/\delta)}{n}} + \frac{\log(1/\delta)}{n}\big) =:   \|A\|_2^2 \tilde v_n(\delta).
\end{align*}
\end{lem}
\begin{proof}
Let $A \in \mathcal{R}(k)$. We have
$$\frac{1}{\sqrt{n}}\mathcal X A \sim \mathcal N(0, \frac{1}{n}\|A\|_2^2 I_n),$$
where $I_n$ is the $n \times n$ identity matrix. This implies that
$$\frac{1}{n}\|\mathcal X A\|_2^2 \sim \frac{1}{n}\|A\|_2^2 \chi^2_{n} = \frac{1}{n}\|A\|_2^2 \sum_{i\leq n}\chi^2_{1},$$
where $\chi^2_{j}$ is the chi square distribution with $j$ degrees of freedom. By Bernstein's inequality, we thus have (since the $\chi^2_{1}$ distribution is sub-Gaussian) that, with probability larger than $1-\delta$,
$$|\frac{1}{n} \sum_{i\leq n}\chi^2_{1} - 1| \leq C\big(\sqrt{\frac{\log(1/\delta)}{n}} + \frac{\log(1/\delta)}{n}\big),$$
where $C$ is an universal constant. This implies that, with probability larger than $1-\delta$,
\begin{align*}
\Big|\frac{1}{n}\|\mathcal X A\|_2^2 - \|A\|_2^2 \Big|\leq C\|A\|_2^2\big(\sqrt{\frac{\log(1/\delta)}{n}} + \frac{\log(1/\delta)}{n}\big) =: \|A\|_2^2 \tilde v_n(\delta).
\end{align*}
This concludes the proof.
\end{proof}
Combining Lemma~\ref{lem:concer} with Pythagoras's theorem as in the proof of Lemma~\ref{lem:tra}, we have that for any $A,B \in \MM$, with probability larger than $1-\delta$,
$$
\Big|\frac{1}{n}\langle \mathcal X A,  \mathcal X B \rangle - \langle A, B\rangle \Big|
\leq 4 \tilde v_n(\delta/3) \|A\|_2 \|B\|_2.
$$
By a union bound, this implies that with probability larger than $1-\delta$,
\begin{align*}
\|\Delta \|_\infty &= \sup_{ m \leq d, m' \leq d}|\Delta_{m,m'}|\\
&= \sqrt{n} \sup_{ m \leq d, m' \leq d} \left| \frac{1}{n}\langle \mathcal X \mathcal U^{m,m'}, \mathcal X \psi^{r+1} \rangle - \langle \mathcal U^{m,m'}, \psi^{r+1} \rangle  \right|   \\
&\leq \sqrt{n} \| \psi^{r+1}\|_2 \left(4 \tilde v_n(\delta/(3d^2)) \right) \\
&\leq C \sqrt{n} \sqrt{\frac{kd\log(1/\delta)}{n}} \sqrt{\frac{\log(d/\delta)}{n}},
\end{align*}
where $C$ is a universal constant. This concludes the proof (in remarking that the above quantity is arbitrarily small when $kd\log(d) = o(n)$).





\paragraph{Acknowledgements} We would like to thank Richard Nickl, Richard Samworth and Rajen Shah for insightful comments and discussions. Part of this work was produced when AC was in the StatsLab in the University of Cambridge. AC’s work is supported since 2015 by the DFG’s Emmy Noether grant MuSyAD (CA 1488/1-1).


%
%
%

\bibliographystyle{chicago}      
\def\noopsort#1{}

\newpage

\appendix

{\huge Supplementary Material}

\section{Results for the sparse linear regression model}\label{ss:sr}

The method that we proposed and studied in the low rank matrix recovery setting can be adapted and simplified to accommodate another setting : the sparse linear regression setting. We explain how to construct an estimator based on IHT, and prove that the estimator is efficient in $L_2$ and $L_\infty$ norm, and provide the limiting distribution of a simple modification of our estimate.

\subsection{Setup}

We let $B(k):=B_0(0,k)$ be the ``$l_0(\mathbb R^p)$ ball" of radius $k$, i.e. $B(k)$ is the subset of the vectors $u \in \mathbb R^p$ such that $u$ has less than $k$ non-zero coordinates.

Consider the linear model
\begin{equation*}
Y=X\theta + \epsilon,
\end{equation*}
where $X$ is a $n \times p$ matrix, the signal vector $\theta \in \rr^p$ is $k$-sparse  ($\theta\in B(k)$), and $\epsilon \in \mathbb R^n$ is an i.i.d.~vector of Gaussian white noise, i.e.~$\epsilon \sim \mathcal N(0, I_n)$ (as in the matrix regression, we do not need the Gaussian assumption and our results hold with sub-Gaussian independent noise), and $p\gg n$. We denote the sample covariance matrix by $\hat \Sigma = \frac{1}{n}X^T X \in \rr^{p \times p}$.



\begin{ass}\label{ass:matrix}
Let $K \leq p$. We assume that there exists a matrix $V$ such that for any $k \leq 2K$, there exists a constant $r_k>0$ such that
$$\sup_{u \in B(k)} \frac{\| V \hat \Sigma u - u \|_{\infty}}{\|u\|_{\infty}} \leq r_k.$$
\end{ass}
\begin{remark}\label{rem:ass}
Suppose $X$ is from a distribution whose covariance matrix is $\Sigma \in \mathbb{R}^{p \times p}$.
Let the minimum eigenvalue $\sigma_{\text{min}}(\Sigma) \geq C_{\text{min}}>0$ and the maximum eigenvalue $\sigma_{\text{max}}(\Sigma) \leq C_{\text{max}} < \infty$ and $\text{max}_{i \in [p]} \Sigma_{ii} \leq 1$. 
Assume that $X \Sigma^{-1/2}$ has independent sub-Gaussian rows with zero mean and sub-Gaussian norm $\|\Sigma^{-1/2}X_1\|_{\psi_2} = \kappa$. Then from the paper \citep{javanmard}, for $n \geq C_{\text{min}}\log p/(4e^2 C_{\text{max}}\kappa^4)$, with probability larger than $1-2p^{-c_2}$ with $c_2 \equiv C_{\text{min}}/(24e^2 \kappa^4 C_{\text{max}})$, there exists a computationally feasible $V$ such that  
\begin{equation}\label{eq:pseudoinv}
\| V \hat \Sigma - I \|_\infty \leq \sqrt{\frac{\log p}{n}}
\end{equation}
 holds.
In this case, we can take $r_k  = k \sqrt{\frac{\log p}{n}}$.
\end{remark}



\subsection{Method}

This algorithm takes again three parameters : $\delta, K$ and $B$. We have the same interpretation for $\delta$  as in the matrix regression setting, $K$ is an upper bound on two times the sparsity of $\theta$ (again, it does not need to be tight as long as $r_K$ is small enough), and $B$ is a loose bound on the $L_{\infty}$ norm of $\theta$.

First, we set the initial values for the estimator $\hat \theta^0$ and the threshold $T_0$ such that 
$$\hat \theta^0 =0, \ \ \ T_0=B.$$

Then we update thresholds in each iteration $r \in \mathbb N^*$, by 
$$T_r = 2r_K T_{r-1} + \upsilon,$$
where $\upsilon = 2\sqrt{M\frac{\log(p/\delta)}{n}}$ where $M =\max \mathrm{diag}(V \hat \Sigma V^T)$.
Recall that the pseudo inverse $V$ of $\hat \Sigma$ and $r_K$ are taken from Assumption \ref{ass:matrix}.
 
Set now recursively, 
$$\hat \alpha^r = \lfloor\frac{1}{n} V X^T(Y - X\hat \theta^{r-1}) \rfloor_{T_r},$$
and
\begin{equation}\label{eq:estimator}
\hat \theta^r = \hat \theta^{r-1} + \hat \alpha^r.
\end{equation}
This procedure provides a sequence of estimates, and as we will prove in the next subsection, this sequence is with high probability close to the true $\theta$ as soon as $r$ is of order $\log(n)$ (see Theorems~\ref{th:mainthm} and~\ref{thm:asymnorm}).

\begin{remark}[Iterative hard thresholding (IHT)]\label{rem:iht}
 The proposed method modifies iterative algorithms \citep[see e.g.][]{blumensath2009, needell2009}. The usual (normalised) IHT algorithm updates the estimate using $\hat \theta^r = P_k(\hat \theta^{r-1}+ w^{r-1}X^T (Y-X\hat \theta^{r-1}))$ where $P_k$ is a hard thresholding operator that keeps the largest $k$ elements of a vector and $w^{r-1} \in \rr$ is a stepsize that can have the interpretation of a Gradient step when it is much smaller than $1$. 
The difference is in the thresholding; we update thresholds while they pick the largest $k$ values after adjusting the added parts. 
Most importantly, previous works on this estimator only considered the case of a deterministic (small) noise, so their analysis is not applicable in our model where the noise is stochastic.  
\end{remark}

\subsection{Main results}

We now provide a theorem that guarantees that the estimate $\hat \theta^r$ in (\ref{eq:estimator}) has an optimal $L_\infty$ risk after $O(\log(n))$ iterations.
\begin{thm}\label{th:mainthm}
Assume that Assumption~\ref{ass:matrix} is satisfied and that $2r_K <1$. Let $r = \log(n)/\log(1/(2r_K)) \approx O(\log(n))$. We have that with probability larger than $1-\delta$, for any $k \leq K/2$,
\begin{equation*}
\sup_{\theta \in B(k), \|\theta\|_{\infty} \leq B} \| \theta - \hat \theta^r\|_\infty \leq  C_0 \sqrt{\frac{M\log(p/\delta) }{n}},
\end{equation*}
where $C_0 = (B+\frac{2}{1-2r_k})$ and $M = \max \mathrm{diag}(V \hat \Sigma V^T) $.
\end{thm}

\begin{remark}\label{rem:ass2}
If the design is obtained as in Remark~\ref{rem:ass}, then as long as $K = o(\sqrt{n/\log(p)})$, with high probability the assumptions of Theorem~\ref{th:mainthm} will hold.
\end{remark}

Theorem \ref{th:mainthm} provides two side results---$L_2$ convergence rates and asymptotic normality. The first corollary is immediately obtained from the fact that for any $\theta \in B(k)$, $\|\theta\|_2 \leq \sqrt{k} \| \theta \|_\infty$.
\begin{cor}
 Suppose that the same assumptions and notation used in Theorem \ref{th:mainthm} hold.
We have that with probability larger than $1-\delta$, for any $k \leq K/2$
$$
\sup_{\theta \in B(k), \|\theta\|_\infty \leq B} \| \hat \theta^r  - \theta\|_2 \leq C_0 \sqrt{\frac{kM \log (p/\delta)}{n}}.
$$
\end{cor}

To prove asymptotic normality, we slightly modify the estimator defined in Theorem \ref{th:mainthm}. This is similar to the de-sparsified LASSO by \cite{vandegeer2014} in the sense that we also use a de-sparsified version of our estimator.
 Consider the estimator $\hat \theta^r$ of Theorem~\ref{th:mainthm} (with the same $r= \log(n)/\log(1/(2r_K))$) and $V$ in Assumption \ref{ass:matrix}, and define 
\begin{equation}\label{eq:newestimator}
\hat \theta := \hat \theta^r + \frac{1}{n}V  X^T( Y-  X \hat \theta^r).
\end{equation}

\begin{thm}\label{thm:asymnorm}
Suppose that the same assumptions and notation used in Theorem \ref{th:mainthm} hold.
Then, writing $Z:=\frac{1}{\sqrt{n}} V X^T \epsilon$ and $\Delta := \sqrt{n}(I-V \hat \Sigma)(\hat \theta^r - \theta)$, we have
\begin{equation}\label{eq:split}
\sqrt{n}(\hat \theta - \theta) = \Delta +Z
\end{equation}
where $Z|X \sim \mathcal N(0, \frac{1}{n}V \hat \Sigma V^T)$. If $r_K = o(1)$ (e.g.~for designs as in Remark~\ref{rem:ass}, we have $r_K = O(K\sqrt{(\log p)/n})$ so it suffices that  $K = o(\sqrt{n}/\log p)$) then we also have
$$\| \Delta \|_\infty = o_{\mathbb{P}} (1).$$
\end{thm}

The estimate we provide has similar properties as in~\cite{javanmard, vandegeer2014}.

\subsection{Proof of Theorem~\ref{th:mainthm}}

We have
\begin{align}
\|(\theta- \hat \theta^{r-1})-\frac{1}{n}VX^T&(Y - X\hat \theta^{r-1}) \|_{\infty} = \| (\theta- \hat \theta^{r-1}) - \frac{1}{n}VX^T(X\theta + \epsilon - X\hat \theta^{r-1})\|_{\infty}\nonumber\\
&= \| (\theta- \hat \theta^{r-1}) - V \hat \Sigma(\theta - \hat \theta^{r-1}) - \frac{1}{n}VX^T\epsilon\|_{\infty}\nonumber\\
&\leq \|(\theta- \hat \theta^{r-1}) -V \hat \Sigma(\theta - \hat \theta^{r-1})\|_{\infty} +  \|\frac{1}{n}VX^T\epsilon\|_{\infty}.\label{eq:hipoule}
\end{align}

Since $\epsilon \sim \mathcal N(0, I_n)$, we know that
$$
\frac{1}{n}VX^T\epsilon
\sim \mathcal N(0, \frac{1}{n} V \hat \Sigma V^T).$$

By an union bound (with Hoeffding's inequality) we know that with probability larger than $1-\delta$
\begin{align}\label{eq:poule2}
\|\frac{1}{n}VX^T\epsilon\|_{\infty} \leq 2\sqrt{M\frac{\log(p/\delta)}{n}} = \upsilon.
\end{align}
Let $\xi$ be the event of probability $1-\delta$ where the previous equation is satisfied.

We have by Assumption~\ref{ass:matrix} if $\theta - \hat \theta^{r-1}$ is $k$ sparse
\begin{align}\label{eq:poule3}
\| V \hat \Sigma (\theta - \hat \theta^{r-1}) - (\theta - \hat \theta^{r-1})\|_{\infty} \leq r_k\|\theta - \hat \theta^{r-1}\|_{\infty}.
\end{align}

Combining Equations~(\ref{eq:hipoule}), (\ref{eq:poule2}) and~(\ref{eq:poule3}) implies that on $\xi$, if if $\theta - \hat \theta^{r-1}$ is $k$ sparse
\begin{align}\label{eq:superpoule}
\|(\theta- \hat \theta^{r-1})-\frac{1}{n}VX^T&(Y - X\hat \theta^{r-1}) \|_{\infty} \leq  r_k\|\theta - \hat \theta^{r-1}\|_{\infty} + \upsilon.
\end{align}

We are going to prove by induction that on $\xi$,
$$\|\theta- \hat \theta^r\|_{\infty} \leq 2 T_r,$$
and that the support of $\hat \theta^r$ is included in the support of $\theta$.

\paragraph{1. Initialisation:} Consider $r = 0$. Since $\hat \theta^0 = 0$, its support is included in the support of $\theta$. Moreover, by definition of $B$, we have that
$$\|\theta - \theta^0\|_{\infty} = \|\theta \|_{\infty} \leq B \leq 2T_0.$$
This concludes the proof for $r = 0$.


\paragraph{2. Induction step:} Assume that for a given $r$, on $\xi$
$$\|\theta- \hat \theta^{r}\|_{\infty}  = \|(\theta- \hat \theta^{r-1}) - \hat \alpha^r\|_{\infty} \leq 2T_r.$$
We moreover assume that the support of $\hat \theta^{r}$ is contained in the support of $\theta$, which implies that it is $k$ sparse.

By Equation~(\ref{eq:superpoule}) we know that on $\xi$, since $\theta - \hat \theta^{r}$ is $k$ sparse
\begin{align}
\|\frac{1}{n}V X^T(Y - X\hat \theta^{r}) - (\theta- \hat \theta^{r})\|_{\infty} &\leq  r_k\|\theta - \hat \theta^{r}\|_{\infty} + \upsilon \nonumber\\
&\leq 2r_k T_r + \upsilon \leq T_{r+1}, \label{equ}
\end{align}
since $T_{r+1} = 2r_KT_r + \upsilon$. Since $\hat \alpha^{r+1} = \lfloor\frac{1}{n} V X^T(Y - X\hat \theta^{r})\big\rfloor_{T_{r+1}}$, we have that on $\xi$, by Equation (\ref{equ}), all the coordinates $j$ of $\hat \alpha^{r+1}$ such that $(\theta- \hat \theta^{r})_j = 0$ are set to $0$. This implies that the support of $\hat \alpha^{r+1}$ (and thus the support of $\hat \theta^{r+1} = \hat \theta^r +\hat \alpha^{r+1}$) is included in the support of $\theta$ on $\xi$. Therefore, $\hat \alpha^{r+1}$ is $k-$sparse on $\xi$. Also, still since $\hat \alpha^{r+1} = \lfloor\frac{1}{n} V X^T(Y - X\hat \theta^{r})\big\rfloor_{T_{r+1}}$, we have that
$$\|\frac{1}{n}VX^T(Y - X\hat \theta^{r})  - \hat \alpha^{r+1}\|_{\infty} \leq T_{r+1},$$
and this implies together with Equation (\ref{equ}) that on $\xi$, we have
\begin{align*}
\|\theta - \hat \theta^{r+1}\|_{\infty} = \|(\theta- \hat \theta^{r}) - \hat \alpha^{r+1}\|_{\infty} \leq 2T_{r+1}.
\end{align*}
This concludes the proof for $r+1$.

The induction is complete, and we have that the previous equation holds for all $r\geq 1$. It is equivalent to the fact that on $\xi$ (and thus with probability larger than $1-\delta$), for all $r\geq 1$
\begin{equation}\label{eq:hello0}
\|\theta- \hat \theta^{r}\|_{\infty} \leq  2T_{r},
\end{equation}
and the support of $\hat \theta^r$ is included in the support of $\theta$.

\paragraph{3. Study of the sequence $T_r$} The sequence $T_r$ is such that
$$T_r = 2r_K T_{r-1} + \upsilon \quad \mathrm{and} \quad T_0 = B.$$
 A simple induction on this geometric sequence provides that
$$T_{r} = \frac{1}{1-2r_K}\Big[(2r_K)^{r}\big((1-2r_K)B - \upsilon\big) + \upsilon\Big]\leq (2r_K)^{r}B + \upsilon/(1-2r_K).$$

\paragraph{4. Conclusion} Let $r = -\log(n)/\log(2r_K) \approx O(\log(n))$, since $2r_K <1$ and is a constant. We have by Equation~\eqref{eq:hello0} and by the recursion on $T_r$ that on $\xi$
\begin{equation}\label{eq:hello}
\|\theta- \hat \theta^{r}\|_{\infty} \leq   \frac{B}{n} + \frac{\upsilon}{1-2r_K} \leq \Big(B + \frac{2}{1-2r_K}\Big)\sqrt{M\frac{\log(p/\delta)}{n}}.
\end{equation}

\subsection{Proof of Theorem~\ref{thm:asymnorm}}
By definition,
\begin{align*}
\sqrt{n}(\hat \theta - \theta) &= \sqrt{n}\left((\hat \theta_r-\theta) + \frac{1}{n}V X^T ( X \theta -  X \hat \theta_r)  + \frac{1}{n} V  X^T \epsilon\right)\\
&=  \sqrt{n}\Big((\hat \theta_r-\theta) - V\hat \Sigma (\hat \theta_r-\theta )\Big)+\frac{1}{\sqrt{n}} V X^T \epsilon
=\Delta +Z.
\end{align*}
Given $X$, we know that $Z$ is a linear function of the Gaussian vector $\epsilon$, thus
$$
Z|X \sim N(0,  V\hat \Sigma V^T).
$$
Now we prove the bound for $\Delta$. Note that using (\ref{eq:pseudoinv}) and $r_k = O(k\sqrt{(\log p)/n})$, for a sufficiently large $n$, we have a constant $C>0$ such that 
$$
\| \Delta\|_\infty = \sqrt{n} \|  (I-V\hat \Sigma)(\hat \theta_r - \theta) \|_\infty \leq Ck\sqrt{\log p} \| \hat \theta_r - \theta \|_\infty.
$$
Then using the result from Theorem \ref{th:mainthm}, with probability at least $1-\delta$,
we have as long as $k = o(\sqrt{n}/\log p)$
$$
\| \Delta\|_\infty \leq C C_0 k \frac{M\log (p/\delta)}{\sqrt{n}}  \rightarrow 0,
$$
as $n \rightarrow \infty$.

\end{document}